\documentclass[12pt]{amsart}
\usepackage{hyperref,amssymb,mathtools,xcolor}
\usepackage[all]{xy}

\theoremstyle{plain}
\newtheorem{theorem}{Theorem}[section]
\newtheorem{proposition}[theorem]{Proposition}
\newtheorem{corollary}[theorem]{Corollary}
\newtheorem{lemma}[theorem]{Lemma}

\theoremstyle{definition}
\newtheorem{remark}[theorem]{Remark}
\newtheorem{question}[theorem]{Question}

\newtheorem{example}[theorem]{Example}

\newcommand{\abs}[1]{\lvert#1\rvert}
\newcommand{\norm}[1]{\lVert#1\rVert}
\newcommand{\bigabs}[1]{\bigl\lvert#1\bigr\rvert}

\newcommand{\term}[1]{{\textit{\textbf{#1}}}}   
\renewcommand{\mid}{\::\:}

\DeclareSymbolFont{bbold}{U}{bbold}{m}{n}
\DeclareSymbolFontAlphabet{\mathbbold}{bbold}
\def\one{\mathbbold{1}}

\DeclareMathOperator{\Range}{Range}
\DeclareMathOperator{\Span}{span}

\DeclareMathOperator{\FVL}{FVL}
\DeclareMathOperator{\FBL}{FBL}

\renewcommand{\le}{\leqslant}
\renewcommand{\ge}{\geqslant}

\newcommand{\goesu}{\xrightarrow{\mathrm{ru}}} 
\newcommand{\goesur}[1]{\xrightarrow{\mathrm{ru}({#1})}} 
\newcommand{\goesso}{\xrightarrow{\sigma\mathrm{o}}} 

\newcommand{\Xru}{X^{\mathrm{ru}}} 
\newcommand{\Yru}{Y^{\mathrm{ru}}}

\newcommand{\Tru}{T^{\mathrm{ru}}}
\newcommand{\Xruu}{\overline{X}^{\mathrm{ru},1}}

\newcommand{\ucl}[1]{\overline{#1}}
\newcommand{\uadh}[2]{\overline{#1}^{{#2}}}

\newcommand{\Icl}[2]{\overline{I_{#1}^{#2}}^{\norm{\cdot}_{#1}}}

\begin{document}

\title{Relative uniform completion of a vector lattice}

\author{Eugene Bilokopytov}
\email{bilokopy@ualberta.ca}

\address{Department of Mathematical and Statistical Sciences,
  University of Alberta, Edmonton, AB, T6G\,2G1, Canada.}

\address{Department of Mathematics, Toronto Metropolitan University, 350 Victoria Street, Toronto, ON, M5B2K3, Canada.}

\author{Vladimir G. Troitsky}
\email{troitsky@ualberta.ca}

\address{Department of Mathematical and Statistical Sciences,
  University of Alberta, Edmonton, AB, T6G\,2G1, Canada.}

\keywords{vector lattice, relative uniform completion}
\subjclass[2020]{Primary: 46A40}


\date{\today}

\begin{abstract}
  In the paper, we revisit several approaches to the concept of
  uniform completion $\Xru$ of a vector lattice~$X$. We show that many
  of these approaches yield the same result. In particular, if $X$ is a
  sublattice of a uniformly complete vector lattice $Z$ then $\Xru$
  may be viewed as the intersection of all uniformly complete
  sublattices of $Z$ containing~$X$. $\Xru$ may also be constructed
  via a transfinite process of taking uniform adherences in $Z$ with
  regulators coming from the previous adherences. If, in addition, $X$
  is majorizing in $Z$ then $\Xru$ may be viewed as the uniform
  closure of $X$ in~$Z$. We show that $\Xru$ may also be characterized
  via a universal property: every positive operator from $X$ to a
  uniformly complete vector lattice extends uniquely
  to~$\Xru$. Moreover, the class of positive operators here may be
  replaced with several other important classes of operators (e.g.,
  lattice homomorphisms). We also discuss conditions when the uniform
  adherence of a sublattice equals its uniform closure, and present an
  example (based on a construction by R.N.~Ball and A.W.~Hager) where
  this fails.
\end{abstract}

\maketitle

\section{Introduction.}

This paper is intended as a survey on approaches to (relative)
uniform completion of vector lattices. In modern Analysis, there are
several common ways to define the concept of a
\emph{completion}. Let us outline a few. The following
descriptions are rather informal; we put in italic the terms that
allow multiple interpretations. Given a \emph{space} $X$ in a
\emph{category} of spaces, by the completion of $X$ in the category,
one may mean
\begin{itemize}
\item The intersection of all \emph{complete} spaces $Y$ in the category,
  containing~$X$. One often assumes, in addition, that $Y$ is
  contained in some ``ambient'' complete space, which may be
  arbitrary, or may be canonical in some sense.
\item The \emph{closure} of $X$ in a complete space~$Y$. Again, here
  $Y$ could be arbitrary or canonical.
\item The \emph{least} complete space $Y$ that contains $X$ as a
  \emph{subspace}.
\item A complete space $Z$ that satisfies a \emph{universal property}:
  for every complete space~$Y$, every \emph{morphism} $T\colon X\to Y$
  admits a unique extension to a morphism $\widehat{T}\colon Z\to Y$.
\end{itemize}
In ``nice'' categories, these approaches yield the same
completion. For example, this is the case for the norm completion of a
normed space~$X$. One can define the completion of $X$ as the closure
of $X$ in $X^{**}$ or, alternatively, in any Banach space $Y$
containing~$X$. Here $X^{**}$ plays the role of a ``canonical''
ambient space. Equivalently, one can define the norm completion of $X$
as the intersection of all closed (i.e., complete) subspaces of
$X^{**}$ (or of any fixed Banach space $Y$ containing $X$) that
contain~$X$. Clearly, every Banach space that contains $X$ as a
subspace also contains the completion of~$X$, up to an isometry.
The completion satisfies the universal property for
continuous operators and for isometries;
moreover, the completion is characterized by these universal properties.

All the approaches listed above have been used in the literature to
define uniform completions of vector lattices, see, e.g.,
\cite{Veksler:69,Aliprantis:84,Triki:02,Buskes:16,Emelyanov:23,Emelyanov:24}
and it has often be presumed that the results are the same. In fact,
they are not always the same, and even when they are, this is not
always obvious. The main cause of trouble here is the fact that
uniform convergence in a sublattice is, generally, different from
uniform convergence in the entire space. The following examples
illustrate how things may go ``wrong''.

\begin{example}\label{ex:ambient}
  Consider the sequence $x_n=\frac1n e_n$ in~$\ell_\infty$. It is easy
  to see that $(x_n)$ is uniformly null in~$\ell_\infty$. However, the
  same sequence fails to be uniformly null when
  considered as a sequence in~$\ell_1$, where $\ell_1$ is viewed as a
  sublattice of~$\ell_\infty$, as the sequence is not even order
  bounded in~$\ell_1$.
\end{example}

\begin{example}\label{ex:c00}
  Let $X=c_{00}$, the space of all sequences of real numbers with only
  finitely many non-zero entries. It is easy to see that $c_{00}$ is
  uniformly complete. However, viewed as a sublattice of~$c_0$, $X$ is not
  uniformly closed. Its uniform closure in $c_0$ is all of~$c_0$.
\end{example}

Example~\ref{ex:c00} shows that, in general, one cannot define the
uniform completion of $X$ to be the uniform closure of $X$ in an
arbitrary uniformly complete vector lattice that contains $X$ as a
sublattice. We will show in this paper that, nevertheless, many other
``natural'' definitions of a uniform completion make sense and agree.

\medskip

Finally, we would like to mention that the concept of uniform
completion has been studied in the setting of lattice-ordered groups;
see, e.g.,~\cite{Ball:99,Cernak:09,Hager:15}. A somewhat different
approach to uniform completions was undertaken in~\cite{BallHager4}.

\section{Notation and preliminaries.}

We refer the reader to~\cite{Aliprantis:06} for background information
on vector lattices.  All vector lattices in this paper are assumed to
be Archimedean.  Given a net $(x_\alpha)$ in a  vector
lattice $X$ and vectors $x\in X$ and $e\in X_+$, we say that
$(x_\alpha)$ converges to $x$ \term{uniformly relative} to $e$ if for
every $\varepsilon>0$ there exists an index $\alpha_0$ such that
$\abs{x_\alpha-x}\le\varepsilon e$ for all
$\alpha\ge\alpha_0$. Equivalently, $\norm{x_\alpha-x}_e\to 0$, where
\begin{displaymath}
  \norm{z}_e\coloneqq
  \inf\bigl\{\lambda\in\mathbb R_+ \mid \abs{z}\le\lambda u\bigr\}
\end{displaymath}
is a norm on the principal ideal $I_e$ (we take $\norm{z}_e=\infty$ if
$z\notin I_e$).  We write $x_\alpha\goesur{e}x$.  We say that $(x_n)$
converges to $x$ \term{relatively uniformly} or just \term{uniformly}
and write $x_\alpha\goesu x$ if $x_\alpha\goesur{e}x$ for some
$e\in X_+$; we then say that $e$ is a \term{regulator} of the
convergence.

Uniform convergence is, generally, not given by a topology. We refer
the reader to Section~5 in~\cite{BCTW} for an overview of properties of uniform
convergence. Here we collect a few facts that will be most important
for our exposition.

If $x_\alpha\goesur{u}x$ then $x_\alpha\goesur{v}x$ for every
$v\ge u$.  Uniform convergence may often be reduced to sequences in
the following way: if $x_\alpha\goesur{u}x$ then there exists an
increasing sequence $(\alpha_n)$ of indices such that
$\abs{x_{\alpha_n}-x}\le\frac1n u$ for all~$n$; in particular,
$x_{\alpha_n}\goesur{u}x$. Let $A$ be a subset of~$X$. $A$ is
\term{uniformly closed} in $X$ if it contains the limits of all
uniformly convergent nets (or sequences) in~$A$. It is easy to see
that the intersection of any collection of uniformly closed sets is
again uniformly closed. The \term{uniform closure} of $A$ is the
intersection of all uniformly closed subsets of $X$ containing~$A$; we
denote it by~$\ucl{A}$. This is, clearly, the least uniformly closed
subset of $X$ containing~$A$. We denote by $\uadh{A}{1}$ the
\term{uniform adherence} of $A$ defined as the set of all $x\in X$ for
which there exists a net $(x_\alpha)$ in $A$ such that
$x_\alpha\goesu x$ in~$X$. Again, it suffices to consider
sequences. Note that $A\subseteq\uadh{A}{1}\subseteq\ucl{A}$;
furthermore, $A$ is uniformly closed iff $A=\ucl{A}$ iff
$A=\uadh{A}{1}$. Since uniform convergence is not topological, we do
not necessarily have $\uadh{A}{1}=\ucl{A}$; see
Example~\ref{ex:adh-non-cl} below.

One can define $\uadh{A}{\kappa}$ for every ordinal $\kappa$ using
transfinite induction: if $\kappa=\iota+1$ then we put
$\uadh{A}{\kappa}\coloneqq\uadh{\uadh{A}{\iota}}{1}$; if $\kappa$ is a
limit ordinal then we put
$\uadh{A}{\kappa}\coloneqq\bigcup_{\iota<\kappa}\uadh{A}{\iota}$. Note
that $\uadh{A}{\omega_1}=\ucl{A}$. Indeed, suppose that
$x\in\uadh{\uadh{A}{\omega_1}}{1}$. Then $x_n\goesu x$ for some
sequence $(x_n)$ in $\uadh{A}{\omega_1}$. Since
$\uadh{A}{\omega_1}=\bigcup_{\kappa<\omega_1}\uadh{A}{\kappa}$, the
entire sequence $(x_n)$ is contained in $\uadh{A}{\kappa}$ for some
$\kappa<\omega_1$. Then
$x\in\uadh{A}{\kappa+1}\subseteq\uadh{A}{\omega_1}$.

It can also be easily verified that if $Y$ is a sublattice of $X$ then
$\uadh{Y}{1}$ and $\ucl{Y}$ are, likewise, sublattices; see, e.g.,
Proposition~3.13 in~\cite{BCTW}. It follows that $\ucl{Y}$ is the
intersection of all uniformly closed sublattices containing~$Y$.

We would like to emphasize that the concepts introduced above ``depend
on the ambient space''. For example, let $Y$ be a sublattice of $X$ and
$(x_\alpha)$ a net in~$Y$; if
$x_\alpha\goesu 0$ in $Y$ then $x_\alpha\goesu 0$ in $X$ (with the
same regulator), but the converse may be false (see Example~\ref{ex:ambient}).
Furthermore, if $A\subseteq Y$ then the uniform closure of $A$ in
$Y$ may be different from the uniform closure of $A$ in~$X$; take,
e.g., $X=c_0$ and $A=Y=c_{00}$ in Example~\ref{ex:c00}.

A net $(x_\alpha)_{\alpha\in \Lambda}$ in $X$ is \term{uniformly
  Cauchy} if the net $(x_\alpha-x_\beta)$, indexed by~$\Lambda^2$,
converges to zero uniformly.  We say that $X$ is \term{uniformly
  complete} if every uniformly Cauchy net is uniformly convergent.  It
is well known (see, e.g.,~\cite[Proposition~5.2]{BCTW}) that $X$ is
uniformly complete iff $\bigl(I_u,\norm{\cdot}_u\bigr)$ is complete as
a normed space for every $u\in X_+$; in this case, by Krein-Kakutani
Theorem, $\bigl(I_u,\norm{\cdot}_u\bigr)$ is lattice isometric to
$C(K)$ for some compact Hausdorff space~$K$. Every Banach lattice, as
well as every order (or Dedekind) complete vector lattice is uniformly
complete. We write $X^\delta$ for the order (or Dedekind) completion
of~$X$. Being order complete, $X^\delta$ is uniformly complete. The
space $c_{00}$ is order complete and uniformly complete because every
principal ideal in it is finite-dimensional.

\begin{example}\label{ex:adh-non-cl}
  Here is an example of a set $A$ for which
  $\uadh{A}{1}\ne\ucl{A}$.
  Let $X=c_{00}$ and
  \begin{displaymath}
    A=\Bigl\{\tfrac1ne_1+\tfrac1k\sum_{i=2}^ne_i\mid
  n,k\in\mathbb N\Bigr\}.
  \end{displaymath}
  Observe that $0\notin\uadh{A}{1}$. Indeed, otherwise, we can find
  sequences $(n_m)$ and $(k_m)$ in $\mathbb N$ and some $h\in c_{00}$
  such that
  \begin{math}
    \frac{1}{n_m}e_1+\frac{1}{k_m}\sum_{i=2}^{n_m}e_i\le\frac1m h
  \end{math}
  for all~$m$. We can write $h=(h_1,\dots,h_l,0,\dots)$ for some
  $l\in\mathbb N$. For every $m\in\mathbb N$ we have $n_m\le l$ and,
  therefore, $\frac1l e_1\le\frac1m h$, which is a contradiction.

  On the other hand, we have $\frac1n e_1\in\uadh{A}{1}$ for
  every~$n$; it follows that $0\in\uadh{A}{2}\subseteq\ucl{A}$.
\end{example}

\begin{example}\label{COmega-ucompl}
  The space $C(\Omega)$ is uniformly complete for every Hausdorff
  topological space~$\Omega$. Indeed, let $u\in C(\Omega)_+$.  Put
  $\Omega_0=\{u\ne 0\}$. It can be easily verified that $g\in I_u$ iff
  there exists $h\in C_b(\Omega_0)$ such that $g(t)=u(t)h(t)$ for
  every $t\in\Omega_0$ and $g$ vanishes on~$\Omega_0^C$. The map that
  sends $g$ to $h$ is a lattice isometry between
  $\bigl(I_u,\norm{\cdot}_u\bigr)$ and $C_b(\Omega_0)$. Being a Banach
  lattice, $C_b(\Omega_0)$ is complete, hence so is
  $\bigl(I_u,\norm{\cdot}_u\bigr)$.
\end{example}

\section{Uniform convergence and ambient space.}

We mentioned earlier that a net that converges uniformly in a
sublattice $Y$ of a vector lattice $X$ remains uniformly convergent
in~$X$. In this section, we consider a few situations when uniform
convergence passes down to a sublattice, that is, when every net in
$Y$ which converges uniformly in $X$ to some element of $Y$ also
converges uniformly in~$Y$. The limit would be the same (because
uniform limits in $X$ are unique), so we may assume WLOG that the
limit is zero. Throughout this section, $Y$ is a sublattice of a
vector lattice~$X$. The following proposition is straightforward, yet
will be heavily used throughout the paper. Recall that $Y$ is
\term{majorizing} in $X$ if for every $x\in X_+$ there exists $y\in Y$
such that $x\le y$.

\begin{proposition}\label{majorizing}
  If $Y$ is majorizing, then uniform convergence of $X$ passes down to~$Y$.
\end{proposition}

\begin{corollary}
  For a net $(x_\alpha)$ in $X$, $x_\alpha\goesu 0$ in $X$ iff
  $x_\alpha\goesu 0$ in~$X^\delta$.
\end{corollary}

The following result appeared in~\cite{Bilokopytov:23}. We provide a
proof for the convenience of the reader and to fix a minor gap in the
original proof.

\begin{proposition}\label{maj-o-dense}
  Let $Y$ be a majorizing sublattice in~$X$.
  \begin{enumerate}
  \item\label{maj-o-dense-adh} Every $0\le x\in\uadh{Y}{1}$ may be
    expressed as a supremum (in $X$) and a uniform limit of an increasing
    sequence in~$Y_+$.
  \item\label{maj-o-dense-clo} Every $0\le x\in\ucl{Y}$ may be
    expressed as a supremum (in $X$) of an increasing sequence in~$Y_+$.
  \end{enumerate}
\end{proposition}

\begin{proof}
  \eqref{maj-o-dense-adh} Find $(y_n)$ in $Y$ such that
  $y_n\goesur{e}x$ for some regulator $e\in X_+$. Since $Y$ is
  majorizing, we may assume that $e\in Y$. Passing to a subsequence,
  we have $x-\frac1n e\le y_n\le x+\frac1n e$ for every~$n$. It
  follows that $x-\frac2n e\le y_n-\frac1n e\le x$ and, therefore,
  $x-\frac2n e\le \bigl(y_n-\frac1n e\bigr)^+\le x$ for
  every~$n$. Take $z_m=\bigvee_{n=1}^m\bigl(y_n-\frac1n
  e\bigr)^+$. Then $z_m\in Y_+$, $z_m\uparrow$, and
  \begin{math}
    0\le x-z_m\le\frac2m e
  \end{math}
  for every~$m$, hence the sequence $(z_m)$ satisfies the requirements.

  \eqref{maj-o-dense-clo} It suffices to show that every
  $0\le x\in\ucl{Y}$ may be expressed as a supremum (in $X$) of a
  countable subset of~$Y_+$. Since $\ucl{Y}=\uadh{Y}{\kappa}$ for a
  sufficiently large ordinal~$\kappa$, it is enough to prove the
  statement for $\uadh{Y}{\kappa}$ for every~$\kappa$. We use
  transfinite induction. The case $\kappa=1$ follows
  from~\eqref{maj-o-dense-adh}. Suppose that the statement is true for
  all $\iota<\kappa$; we will prove it for $\uadh{Y}{\kappa}$.  Let
  $x\in\uadh{Y}{\kappa}_+$. If $\kappa=\iota+1$ for some $\iota$ then,
  by~\eqref{maj-o-dense-adh}, we find a sequence $(y_m)$ in
  $\uadh{Y}{\iota}_+$ such that $x=\sup y_m$. By induction hypothesis,
  for every $m$ we find a countable set $A_m$ of $Y_+$ such that
  $y_m=\sup A_m$. It follows that $y=\sup\bigcup_{m=1}^\infty
  A_m$. Finally, if $\kappa$ is a limit ordinal then
  $x\in\uadh{Y}{\iota}$ for some $\iota<\kappa$, and the required set
  exists by the induction hypothesis.
\end{proof}

Recall that a sublattice $Y$ of a vector lattice $X$ is \term{super
  order dense} if for every $x\in X_+$ there exists a sequence $(y_n)$
in $Y_+$ such that $y_n\uparrow
x$. Proposition~\ref{maj-o-dense}\eqref{maj-o-dense-clo} essentially
says that $Y$ is super order dense in~$\ucl{Y}$.

\begin{proposition}\label{maj-cl-in-cl}
  Let $Y$ be a majorizing sublattice in~$X$; let $\ucl{Y}$ be the
  uniform closure of $Y$ in~$X$. The uniform closure of $Y$ in
  $\ucl{Y}$ is again~$\ucl{Y}$.
\end{proposition}

\begin{proof}
  Let $Z$ be the uniform closure of $Y$ in~$\ucl{Y}$. It suffices to
  show that $Z$ is uniformly closed in $X$. Suppose that
  $y_\alpha\goesu x$ in $X$ for some net $(y_\alpha)$ in $Z$ and some
  $x\in X$. Then $x\in\ucl{Y}$ because $\ucl{Y}$ is uniformly closed
  in~$X$. Since $Y$ and, therefore, $\ucl{Y}$, are majorizing in~$X$,
  it follows from Proposition~\ref{majorizing} that $y_\alpha\goesu x$
  in~$\ucl{Y}$. Since $Z$ is uniformly closed in~$\ucl{Y}$, we
  conclude that $x\in Z$.
\end{proof}

The next few results do not require that the sublattice is majorizing.
The following lemma has been known; see, e.g., Lemma~1.12
in~\cite{Cernak:09}, Lemma~3.1 in \cite{Kalauch:19}, or Lemma~1 in
\cite{Emelyanov:24}.

\begin{lemma}\label{uo-Cauchy-amb}
  Suppose that $x_n\goesu 0$ in $X$ for some $(x_n)$
  in~$Y$. If $(x_n)$ is uniformly Cauchy in $Y$ then
  $x_n\goesu 0$ in~$Y$.
\end{lemma}

\begin{proof}
  By assumption, there exists $v\in Y_+$ such that for every
  $\varepsilon>0$ there exists $m_0$ such that for all $n\ge m\ge
  m_0$, we have
  \begin{equation}\label{eq:uu-seq-sublat}
    x_m-x_n\in[-\varepsilon v,\varepsilon v]_Y\subseteq
    [-\varepsilon v,\varepsilon v]_X.
  \end{equation}
  By the Archimedean property, order intervals in $X$ are uniformly
  closed. Since $x_n\goesu 0$ in~$X$, passing to the
  limit on $n$ in~\eqref{eq:uu-seq-sublat} yields
  $x_m\in[-\varepsilon v,\varepsilon v]_X$, hence
  $x_m\in[-\varepsilon v,\varepsilon v]_Y$. We conclude that
  $x_n\goesu 0$ in~$Y$.
\end{proof}

The following two results extend Proposition~1.12 in~\cite{Taylor:20}.

\begin{proposition}\label{uconv-cl}
  Suppose that $X$ is uniformly complete and let $(x_k)$ be a sequence
  in~$Y$. If $x_k\goesu x$ in $X$ then $x_k\goesu x$ in~$\uadh{Y}{1}$.
\end{proposition}

\begin{proof}
  Clearly, $x\in\uadh{Y}{1}$. Let $e\in X_+$ such that
  $x_k\goesur{e}x$. Without loss of generality, we may assume that
  $x\in I_e$; otherwise, replace $e$ with $e\vee\abs{x}$. Passing to a
  tail, we may also assume that $(x_k)$ is in~$I_e$. Note that $I_e$
  is a Banach lattice under~$\norm{\cdot}_e$. Let $Z$ be the
  norm closure of $Y\cap I_e$ in
  $\bigl(I_e,\norm{\cdot}_e\bigr)$. Observe that $Z$ is a closed
  sublattice of $\bigl(I_e,\norm{\cdot}_e\bigr)$, $Z\subseteq
  \uadh{Y}{1}$, and $x\in Z$.

  For every $n$ there exists $k_n$ such that
  $\abs{x_k-x}\le \frac{1}{n^3}e$ for all $k\ge k_n$. WLOG,
  $k_n<k_{n+1}$ for every~$n$. Put
  \begin{math}
    v_n=\bigvee_{k=k_n}^{k_{n+1}-1}\abs{x_k-x}.
  \end{math}
  Then $v_n\le\frac{1}{n^3}e$ and, therefore,
  $\norm{v_n}_e\le\frac{1}{n^3}$. It follows that the series
  $w\coloneqq\sum_{n=1}^\infty nv_n$ converges in
  $\bigl(I_e,\norm{\cdot}_e\bigr)$, so that $w$ is in~$Z$, hence
  in~$\uadh{Y}{1}$.  We claim that $x_k\goesur{w}x$. Let
  $n\in\mathbb N$. Take any $k\ge k_n$. Find $m\ge n$ such that
  $k_m\le k< k_{m+1}$. Then
  $\abs{x_k-x}\le v_m\le\frac{1}{m}w\le\frac{1}{n}w$.
\end{proof}

\begin{corollary}\label{uconv-stable}
  Uniform convergence of sequences passes down from a uniformly complete vector
  lattice to a uniformly closed sublattice.
\end{corollary}



We will next prove a variant of Proposition~\ref{maj-cl-in-cl} for
sublattices that are not assumed to be majorizing. As before, $Y$ is
assumed to be a sublattice of a vector lattice~$X$. For
$A\subseteq Y$, we write $\uadh{A}{1}_Y$ and $\ucl{A}_Y$ for the
uniform adherence and, respectively, closure of $A$ in~$Y$. In
particular, $\uadh{A}{1}=\uadh{A}{1}_X$ and $\ucl{A}=\ucl{A}_X$.

\begin{proposition}\label{cl-cl}
  Suppose that $X$ is uniformly complete and $A\subseteq Y$. Then
  $\uadh{A}{1}=\uadh{A}{1}_{\uadh{Y}{1}}$ and
  $\ucl{A}=\ucl{A}_{\uadh{Y}{}}$.
\end{proposition}

\begin{proof}
  If $x\in\uadh{A}{1}$ then there exists a sequence $(y_n)$ in $A$
  such that $y_n\goesu x$ in $X$ and, by Proposition~\ref{uconv-cl},
  in~$\uadh{Y}{1}$. It follows that
  $x\in\uadh{A}{1}_{\uadh{Y}{1}}$. Hence,
  $\uadh{A}{1}\subseteq\uadh{A}{1}_{\uadh{Y}{1}}$. The opposite
  inclusion is straightforward, so we have
  $\uadh{A}{1}=\uadh{A}{1}_{\uadh{Y}{1}}$.

  We will prove by transfinite induction that
  $\uadh{A}{\kappa}=\uadh{A}{\kappa}_{\uadh{Y}{\kappa}}$ for every
  ordinal $\kappa$; this will imply (by
  taking a sufficiently large $\kappa$) that
  $\ucl{A}=\ucl{A}_{\ucl{Y}}$. We have already proved the statement
  for $\kappa=1$. Suppose that we already know that
  \begin{math}
     \uadh{A}{\kappa-1}=\uadh{A}{\kappa-1}_{\uadh{Y}{\kappa-1}}=:B.
  \end{math}
  Since we always have
  \begin{math}
    \uadh{A}{\kappa-1}_{\uadh{Y}{\kappa-1}}\subseteq
    \uadh{A}{\kappa-1}_{\uadh{Y}{\kappa}}\subseteq
    \uadh{A}{\kappa-1},
  \end{math}
  it follows that
  \begin{math}
    \uadh{A}{\kappa-1}_{\uadh{Y}{\kappa}}=B.
  \end{math}
  By the first part of the proof,
  \begin{displaymath}
    \uadh{A}{\kappa}=\uadh{B}{1}
    =\uadh{B}{1}_{\uadh{\uadh{Y}{\kappa-1}}{1}}
    =\uadh{B}{1}_{\uadh{Y}{\kappa}}
    =\uadh{\Bigl(\uadh{A}{\kappa-1}_{\uadh{Y}{\kappa}}\Bigr)}{1}_{\uadh{Y}{\kappa}}
    =\uadh{A}{\kappa}_{\uadh{Y}{\kappa}}.
  \end{displaymath}
  Suppose now that $\kappa$ is a limit ordinal and
  $\uadh{A}{\iota}=\uadh{A}{\iota}_{\uadh{Y}{\iota}}$ whenever
  $\iota<\kappa$. Since we always have
  \begin{math}
    \uadh{A}{\iota}_{\uadh{Y}{\iota}}\subseteq
    \uadh{A}{\iota}_{\uadh{Y}{\kappa}}\subseteq
    \uadh{A}{\iota},
  \end{math}
  it follows that
  \begin{math}
    \uadh{A}{\iota}_{\uadh{Y}{\kappa}}=
            \uadh{A}{\iota}.
  \end{math}
  We conclude that
  \begin{math}
    \uadh{A}{\kappa}_{\uadh{Y}{\kappa}}
    =\bigcup_{\iota<\kappa}\uadh{A}{\iota}_{\uadh{Y}{\kappa}}
    =\bigcup_{\iota<\kappa}\uadh{A}{\iota}
    =\uadh{A}{\kappa}.
  \end{math}
\end{proof}

The following proposition extends Lemma~5.4 in~\cite{Bilokopytov:23}.

\begin{proposition}\label{pass-down}
  TFAE:
  \begin{enumerate}
  \item\label{pass-down-pass} Uniform convergence passes down to~$Y$;
  \item\label{pass-down-bdd} A subset of $Y$ that is order bounded in
    $X$ is also order bounded in~$Y$.
  \end{enumerate}
\end{proposition}

\begin{proof}
  \eqref{pass-down-pass}$\Rightarrow$\eqref{pass-down-bdd} Let
  $A\subseteq Y$ and $u\in X_+$ such that $A\subseteq[-u,u]$. Let
  $\Lambda=(0,1]\times A$. Order $\Lambda$ as follows: $(r,a)\le(s,b)$
  if $s\le r$. This is a directed pre-order on~$\Lambda$. If
  $\alpha=(r,a)\in\Lambda$, put $x_\alpha=ra$. Then $(x_\alpha)$ is a
  net. It is easy to see that $\abs{x_\alpha}\le ru$. It follows that
  $x_\alpha\goesu 0$ in $X$ and, therefore, in~$Y$. In particular, the
  net is eventually order bounded in~$Y$. Since every tail of the net
  contains $rA$ as a subset for some $r>0$, we conclude that $A$ is
  order bounded in~$Y$.

  \eqref{pass-down-bdd}$\Rightarrow$\eqref{pass-down-pass} Suppose
  $x_n\goesur{e}0$ for some sequence $(x_n)$ in $Y$ and $e\in
  X_+$. Then $\norm{x_n}_e\to 0$ and
  $\frac{\abs{x_n}}{\norm{x_n}_e}\le e$ for every $n$, so the sequence
  $\bigl(\frac{\abs{x_n}}{\norm{x_n}_e}\bigr)$ is bounded above by some
  $u\in Y$. It follows that $\abs{x_n}\le\norm{x_n}_eu$, so that
  $x_n\goesur{u}0$.
%
%
%
  %
%
%
\end{proof}

\begin{remark}
  The equivalence of \eqref{pass-down-pass} and \eqref{pass-down-bdd}
  in Proposition~\ref{pass-down} may also be deduced from the general
  theory of bornological convergences. Theorem~7.9 in~\cite{BCTW}
  asserts that a linear operator between bornological vector spaces is
  bounded iff it is continuous with respect to the convergences
  induced by the bornologies. Consider two bornologies on~$Y$: the
  bornology $\mathcal B$ of order bounded sets in $Y$ and the
  bornology $\mathcal C$ of those sets in $Y$ that are order bounded
  in~$X$. Condition \eqref{pass-down-bdd} means that the identity map
  on $Y$ is bounded from $\mathcal C$ to $\mathcal B$. It is easy to
  see that the convergence $\mu_{\mathcal B}$ on $Y$ induced by
  $\mathcal B$ is the uniform convergence in~$Y$, while the
  convergence $\mu_{\mathcal C}$ on $Y$ induced by $\mathcal C$ in the
  restriction of the uniform convergence on $X$
  to~$Y$. Condition~\eqref{pass-down-pass} says that the identity map
  on $Y$ is $\mu_{\mathcal C}$-to-$\mu_{\mathcal B}$ continuous.
\end{remark}

By Lemma~5.4 in~\cite{Bilokopytov:23}, if $Y$ is an ideal in $X$ then
conditions \eqref{pass-down-pass} and \eqref{pass-down-bdd} in
Proposition~\ref{pass-down} imply that $Y$ is a projection band.

\begin{question}
  In view of Proposition~\ref{pass-down}, it is clear that
  Corollary~\ref{uconv-stable} fails for nets. Does it remain valid
  for nets with countable index sets?
\end{question}

If $Y$ is a sublattice of $X$ and $u\in Y$, we write $I^Y_u$ for the
principal ideal generated by $u$ in~$Y$. We write $\Icl{u}{Y}$ for the
closure of $I_u^Y$ in $I_u^X$ under~$\norm{\cdot}_u$.

\begin{proposition}\label{loc-adh}
  Let $Y$ be a sublattice of a vector lattice $X$ and $x\in X$. TFAE:
  \begin{enumerate}
  \item\label{loc-adh-adh} $x$ is a uniform limit in $X$ of a sequence
    in $Y$ with a regulator in~$Y$;
  \item\label{loc-adh-I} There exists $u\in Y_+$ such that
    $x\in\Icl{u}{Y}$;
  \item\label{loc-adh-id} $x$ belongs to the uniform adherence of $Y$
    in $I(Y)$.
  \end{enumerate}
\end{proposition}

\begin{proof}
  \eqref{loc-adh-adh}$\Rightarrow$\eqref{loc-adh-I} There exists a
  sequence $(y_n)$ in $Y$ and $v\in Y_+$ such that
  $\abs{y_n-x}\le\frac1n v$ for all $n\in\mathbb N$. Then
  $\abs{x}\le\abs{y_1}+v$; denote the latter by~$u$. Clearly,
  $x\in I_u^X$.  For every $n$, we have $\abs{y_n-x}\le\frac1n u$, so
  that $\norm{y_n-x}_u\to 0$. Also, it follows from
  $\abs{y_n}\le\abs{x}+\frac{1}{n}v\le 2u$ that $(y_n)$ is in~$I^Y_u$.

  \eqref{loc-adh-I}$\Rightarrow$\eqref{loc-adh-id} Clearly,
  $x\in I_u^X\subseteq I(Y)$. Moreover, there exists a sequence
  $(y_n)$ in $I_u^Y$ (hence in $Y$) such that $\norm{y_n-x}_u\to
  0$. It follows that $y_n\goesur{u}x$ in $I(Y)$.

   \eqref{loc-adh-id}$\Rightarrow$\eqref{loc-adh-adh} is straightforward.
\end{proof}

It is easy to see that if $u\le v$ in $Y_+$ then
$\Icl{u}{Y}\subseteq\Icl{v}{Y}$.

\begin{corollary}\label{IuX-cofin}
  Let $Y$ be a sublattice of a vector lattice $X$ and $A\subseteq Y$
  such that for every $u\in Y_+$ there exists $v\in A$ with $u\le
  v$. Then $\bigcup_{v\in A}\Icl{v}{Y}$
  is the uniform adherence of $Y$ in $I(Y)$.
\end{corollary}

In particular, if $I_u^Y$ is closed in $I_u^X$ under $\norm{\cdot}_u$,
for every $u\in A$, then $Y$ is uniformly closed in $I(Y)$.

\begin{proposition}\label{ideals-up}
  Let $Y$ be a sublattice of a vector lattice $X$. If
  $\Icl{u}{Y}$ is an ideal in
  $\Icl{v}{Y}$ whenever $u\le v$ in $Y_+$ then
  $\bigcup_{v\in Y_+}\Icl{v}{Y}$ is the uniform
  closure of $Y$ in $I(Y)$. Moreover, for every $u\in Y_+$ the ideal
  generated by $u$ in that closure is~$\Icl{u}{Y}$.
\end{proposition}

\begin{proof}
  Put $Z\coloneq\bigcup_{v\in Y_+}\Icl{v}{Y}$. It is easy to see that
  $Z\subseteq I(Y)$.  For every $u\in Y_+$, we have
  \begin{math}
    Z=\bigcup_{v\ge u}\Icl{v}{Y}.
  \end{math}
  It follows from the assumption that
  $\Icl{u}{Y}$ is an ideal in~$Z$.

  We claim that $\Icl{u}{Y}=I_u^Z$. On one hand, if $z\in\Icl{u}{Y}$
  then $z\in Z$ by the definition of $Z$ and $z\in I_u^X$ by the
  definition of~$\Icl{u}{Y}$; it follows that $z\in I_u^Z$, hence
  $\Icl{u}{Y}\subseteq I_u^Z$. Since $I_u^Z$ is the least ideal in $Z$
  containing $u$, we have $I_u^Z\subseteq\Icl{u}{Y}$.

  We conclude that $I_u^Z$ is closed in $I_u^X$
  under~$\norm{\cdot}_u$. Now since $u$ was arbitrary, by the comment
  before the proposition we conclude that $Z$ is uniformly closed in
  $I(Y)$, thus equal to the uniform closure of $Y$ in $I(Y)$ by
  Corollary~\ref{IuX-cofin}.
\end{proof}

\section{Uniform completion vs uniform closure.}

\begin{proposition}[Proposition~2.2 in \cite{Kalauch:19}]\label{u-compl-closed}
  Let $Y$ be a sublattice of a uniformly complete vector
  lattice~$X$. If $Y$ is uniformly closed in $X$ then $Y$ is uniformly
  complete. The converse holds if $Y$ is majorizing.
\end{proposition}

\begin{proof}
  Suppose that $Y$ is uniformly closed. Let $(x_n)$ be a uniformly
  Cauchy sequence in~$Y$. Clearly, it remains uniformly Cauchy when
  viewed as a sequence in $X$. Hence $x_n\goesu x$ for some $x\in
  X$. Since $Y$ is uniformly closed, we have $x\in Y$. We now have
  $x_n\goesu x$ in $Y$ by Lemma~\ref{uo-Cauchy-amb} or by
  Corollary~\ref{uconv-stable}.

  Suppose now that $Y$ is majorizing in $X$ and uniformly
  complete. Let $x_n\goesu x$ in $X$ for some $(x_n)$ in $Y$ and
  $x\in X$. Then $(x_n)$ is uniformly Cauchy in~$X$. Since $Y$ is
  majorizing, $(x_n)$ is uniformly Cauchy in~$Y$. Then $x_n\goesu y$
  in $Y$ for some $y\in Y$. It follows that $x_n\goesu y$ in~$X$;
  hence $x=y$ and, therefore, $x\in Y$.
\end{proof}

Example~\ref{ex:c00} shows that the ``majorizing''
assumption in the preceding proposition cannot be removed.  The
following result is similar to Lemma~3 in~\cite{Emelyanov:24}.

\begin{proposition}\label{ucompl-int}
  The intersection of any non-empty family of uniformly complete
  sublattices of a vector lattice is again uniformly complete.
\end{proposition}

\begin{proof}
  Let $\mathcal A$ be a non-empty family of uniformly complete
  sublattices of a vector lattice~$X$, and let $Z=\bigcap\mathcal A$.
  Let $(x_n)$ be a uniformly Cauchy sequence in~$Z$; in particular,
  there is a regulator in~$Z$. Let $Y\in\mathcal A$. It follows from
  $Z\subseteq Y$ that $(x_n)$ is still uniformly Cauchy in~$Y$. Since
  $Y$ is uniformly complete, $x_n\goesu x$ in $Y$ for some $x\in Y$
  (and with a regulator in~$Y$). Therefore, $x_n\goesu x$
  in~$X$. Since uniform limits are unique, we conclude that $x$ does
  not depend on~$Y$. Then $x$ belongs to every member of~$\mathcal A$,
  hence to~$Z$. Since $x_n\goesu x$ in~$Y$, Lemma~\ref{uo-Cauchy-amb}
  yields $x_n\goesu x$ in~$Z$.
\end{proof}

Let $X$ be a vector lattice. Let $\Xru$ be the
intersection of the family of all uniformly complete sublattices of
$X^\delta$ that contain~$X$. Then $\Xru$ itself is a sublattice of
$X^\delta$ containing $X$. We claim that $\Xru$ is uniformly
complete. On one hand, this follows immediately from
Proposition~\ref{ucompl-int}. Alternatively, we may observe that
sublattices in the family contain $X$ and, therefore, are majorizing
in~$X^\delta$, hence uniformly closed in $X^\delta$ by
Proposition~\ref{u-compl-closed}; this implies that $\Xru$ is
uniformly closed and, therefore, uniformly complete. We call $\Xru$
the \term{relatively uniform} or just \term{uniform completion} of~$X$
(we use $\Xru$ rather than $X^u$ to distinguish it from the universal
completion). Our definition of $\Xru$ is similar to that
in~\cite{Veksler:69}.

\begin{proposition}\label{Xru-cl-Xdelta}
  $\Xru$ equals the uniform closure $\ucl{X}$ of $X$ in~$X^\delta$.
\end{proposition}

\begin{proof}
  For a sublattice $Y$ of $X^\delta$ such that $X\subseteq Y$, we know
  that $Y$ is uniformly complete iff it is uniformly closed in
  $X^\delta$ by Proposition~\ref{u-compl-closed}. Hence,
  $\Xru$ is the intersection of all uniformly closed
  sublattices of $X^\delta$ containing~$X$; the latter equals~$\ucl{X}$.
\end{proof}

In the preceding discussion, the order completion $X^\delta$ of $X$
plays a critical role. Would it be possible to find an equivalent
definition of $\Xru$ that avoids explicit use of~$X^\delta$?  Can we
replace $X^\delta$ with another ``ambient space''?
Example~\ref{ex:c00} shows that we generally cannot define $\Xru$ as
the closure of $X$ in an arbitrary uniformly complete vector lattice
containing~$X$. We will show later that, in some sense, we can view
$\Xru$ as the intersection of all uniformly complete vector lattices
that contain $X$ as a sublattice.

Using $X\subseteq\Xru\subseteq X^\delta$ and
Proposition~\ref{maj-o-dense}\eqref{maj-o-dense-clo}, we immediately
get the following:

\begin{proposition}\label{ru-maj}
  $X$ is order dense (moreover, super order dense) and majorizing in~$\Xru$.
\end{proposition}

Propositions~\ref{maj-cl-in-cl} and~\ref{Xru-cl-Xdelta} yield the following:

\begin{corollary}\label{ru-clos-ru}
  The uniform closure of $X$ in $\Xru$ is~$\Xru$.
\end{corollary}

\begin{example}
  $\Xru$ need not agree with~$X^\delta$; just take any Banach lattice
  which is not order complete. Uniform completion of a normed lattice
  need not agree with its norm completion: take~$c_{00}$ (which is
  uniformly complete) viewed as a norm dense subspace of~$c_0$. We
  will characterize in Theorem~\ref{ncompl} when uniform completion
  agrees with norm completion.
\end{example}

\begin{corollary}\label{hier}
  If $X\subseteq Y\subseteq\Xru$ then $\Yru=\Xru$.
\end{corollary}

\begin{proof}
  It follows from $X\subseteq Y\subseteq X^\delta$ that
  $Y^\delta=X^\delta$. The result now follows
  from Proposition~\ref{Xru-cl-Xdelta}.
\end{proof}

\section{Universal properties.}

Recall that a linear operator $T\colon X\to Y$ between vector lattices
is order bounded iff it is (relatively) uniformly continuous in the
sense that $x_\alpha\goesu x$ implies $Tx_\alpha\goesu Tx$ for every
net $(x_\alpha)$ in~$X$; see, e.g., Theorem~10.3
in~\cite{OBrien:23}. It follows that pre-images of uniformly closed
sets under $T$ are uniformly closed.

Recall also that lattice homomorphisms are exactly the positive
disjointness preserving operators. A linear operator $T$ between
vector lattices is disjointness preserving iff
$\abs{Tx}=\bigabs{T\abs{x}}$ for every $x$. Furthermore, if $T$ is order
bounded and disjointness preserving then $\abs{T}$ exists and satisfies
$\abs{Tx}=\abs{T}\abs{x}$ for all~$x$; see, e.g.,~\cite{Duhoux:82}.

We will consider several classes of operators between vector lattices,
e.g., order bounded operators, order continuous operators, lattice
homomorphisms, etc. We say that the \term{universal property for uniform
  completions} is satisfied for a class of operators if every operator
$T\colon X\to Z$ in the class extends uniquely to an operator
$\widehat{T}\colon\Xru\to Z$ in the same class, provided that $X$ and
$Z$ are vector lattices and $Z$ is
uniformly complete.  The goal of this section is to show that the
universal property is satisfied for many natural classes of
operators. This will be achieved by applying the following theorem in the
special case $Y=\Xru$:

\begin{theorem}\label{univ-ob}
  Suppose that $X$ is a majorizing sublattice of a vector
  lattice~$Y$, and $T\colon X\to Z$ is an order bounded operator from
  $X$ to a uniformly complete vector lattice~$Z$. Then $T$
  extends uniquely to an order bounded operator
  $\widehat{T}\colon\ucl{X}\to Z$, where $\ucl{X}$ is the uniform
  closure of $X$ in~$Y$.
\end{theorem}

\begin{proof}
    By Proposition~\ref{maj-cl-in-cl} we may, without loss of
  generality, assume that $\ucl{X}=Y$.

  We first show uniqueness. Let $S,R\colon Y\to Z$ be two order
  bounded extensions of $T$. Then $S-R$ vanishes on~$X$, i.e.,
  $\ker(S-R)$ is a uniformly closed subspace of $Y$ containing $X$
  and, therefore,~$Y$. It follows that $S=R$.

  To prove existence, let $J\colon Z\hookrightarrow Z^\delta$ be the
  inclusion map. The composition $JT\colon X\to Z^\delta$ is order
  bounded, hence regular. Applying Kantorovich Extension Theorem (see
  Theorem~1.32 on~\cite{Aliprantis:06}) to $(JT)^+$ and $(JT)^-$, we
  conclude that $JT$ extends to a regular (hence, order bounded)
  operator $\widehat{T}\colon Y\to Z^\delta$.  Since $Z$ is uniformly
  complete, it is uniformly closed in $Z^\delta$ by
  Proposition~\ref{u-compl-closed}. It follows that
  $\widehat{T}^{-1}(Z)$ is uniformly closed in~$Y$. As it
  contains~$X$, we conclude that $\widehat{T}^{-1}(Z)=Y$ and,
  therefore, we may view $\widehat{T}$ as an operator from $Y$
  to~$Z$. This operator is still order bounded because $Z$ is
  majorizing in~$Z^\delta$.
\end{proof}

\begin{theorem}\label{univ-other}
  Let $Y$ and $Z$ be vector lattices, let $T\colon Y\to Z$
  be an order bounded operator, and let $X$ be a
  majorizing and uniformly dense (in the sense of closure) sublattice
  of~$Y$. For each of the following properties, $T$ has it iff
  $T_{|X}$ has it: positivity, order continuity,
  disjointness preserving, injectivity plus disjointness preserving,
  as well as any combinations of these properties (including lattice
  homomorphisms).
\end{theorem}

\begin{proof}
  By Proposition~\ref{maj-o-dense}\eqref{maj-o-dense-clo}, $X$ is
  order dense in $Y$ and, therefore, we may view $Y$ as a sublattice
  of~$X^\delta$. It is obvious
  that positivity, disjointness preserving, and injectivity pass from
  $T$ to~$T_{|X}$. If $T$ is order continuous then so
  is $T_{|X}$ because $T_{|X}=Tj$, where $j\colon X\hookrightarrow Y$
  is the inclusion, and $j$ is order continuous because $X$ is order
  dense and, therefore, regular in~$Y$.

  If $T_{|X}$ is positive then $TX_+\subseteq Z_+$, so that
  $T^{-1}Z_+$ is a uniformly closed subset of $Y$ containing~$X$. As
  $Y_+$ is the uniform closure of $X_+$ in~$Y$, we conclude that
  $Y_+\subseteq T^{-1}Z_+$, hence $T$ is positive.

  Suppose $T_{|X}$ is disjointness-preserving. Define a map
  $\Phi\colon Y\to Z$ via
  \begin{math}
    \Phi(y)=\abs{Tx}-\abs{T}\abs{x}.
  \end{math}
  Since $T$ is uniformly continuous, so is~$\Phi$. It
  follows that $\ker\Phi$ is uniformly closed. Since $\ker\Phi$
  contains~$X$, we conclude that $\Phi$ is identically zero on $Y$ and,
  therefore, $T$ is disjointness preserving.

  Suppose that $T_{|X}$ is disjointness preserving and injective. By the
  preceding, $T$ is disjointness preserving.
  Let $0\ne y\in Y$. Since $X$ is order dense in $Y$ by
  Proposition~\ref{maj-o-dense}, there exists $x\in X$ such that
  $0<x\le\abs{y}$. Using~\cite{Duhoux:82} again, we get
  \begin{displaymath}
    0<\abs{Tx}
    =\bigabs{T}x\le\bigabs{T}\abs{y}
    =\bigabs{Ty}.
  \end{displaymath}
  It follows that $Ty\ne 0$.

  Suppose that $T_{|X}$ is order continuous. Since the inclusion
  $J\colon Z\hookrightarrow Z^\delta$ is order continuous, so is
  $JT_{|X}$. By Ogasawara's Theorem (see, e.g., Theorem~1.57
  in~\cite{Aliprantis:06}), $(JT_{|X})^+$ and $(JT_{|X})^-$ are order
  continuous. By Veksler's Theorem (see, e.g., Theorem~1.65
  in~\cite{Aliprantis:06}), they extend to order continuous operators
  from $X^\delta$ to~$Z^\delta$, hence the difference of these
  extensions yields an order continuous extension of $JT_{|X}$ to some
  $S\colon X^\delta\to Z^\delta$. Since $S_{|Y}$ and $JT$ are two
  uniformly continuous operators from $Y$ to $Z^\delta$ which agree on~$X$,
  we have $S_{|Y}=JT$. As order convergences on $Y$ and $Z$ agree with
  those on $X^\delta$ and~$Z^\delta$, respectively
  (see~\cite{Abramovich:05} or Corollary~2.9 in~\cite{Gao:17}), we
  conclude that $JT$ is order continuous as an operator from $Y$ to
  $Z^\delta$ and, therefore, $T$ is order continuous as an operator
  from $Y$ to~$Z$.
\end{proof}

Applying Theorems~\ref{univ-ob} and~\ref{univ-other} in the special
case $Y=\Xru$ yields the main result of this section:

\begin{theorem}\label{univ}
  The universal property for uniform completions is satisfied for the
  following classes of operators: order bounded, order continuous,
  positive, order bounded disjointness preserving, and injective order
  bounded disjointness preserving operators, as well as any
  intersection of these classes (including lattice homomorphisms).
\end{theorem}

We will later show in Example~\ref{ex:pos-inj} that Theorem~\ref{univ}
fails for the class of injective operators: there exists a
positive injective order continuous operator whose extension fails to
be injective.

\begin{question}
  Does the universal property hold for sequentially uniformly
  continuous operators?
\end{question}

For each class of operators in Theorem~\ref{univ}, $\Xru$ is the
unique space that satisfies the universal property in the following
sense. Let $j\colon X\hookrightarrow\Xru$ be the inclusion map. Note
that $j$ itself belongs to all the classes. Suppose also that $X$ is a
sublattice of some uniformly complete vector lattice $Z$ such that the
pair $(X,Z)$ satisfies the universal property for some class; let
$i\colon X\hookrightarrow Z$ be the inclusion map. Suppose that $i$
itself belongs to the class. Let $\widehat{j}\colon Z\to\Xru$ and
$\widehat{i}\colon\Xru\to Z$ be the extensions given by the universal
properties. Then $\widehat{j}\widehat{i}\colon\Xru\to\Xru$
extends~$j$, hence, by uniqueness,
$\widehat{j}\widehat{i}=id_{\Xru}$. Analogously,
$\widehat{i}\widehat{j}=id_Z$. It follows that $Z$ is lattice
isomorphic to~$\Xru$.

In~\cite{Aliprantis:84}, uniform completion is defined via the
universal property for lattice homomorphisms. We now see that this
definition is equivalent to the earlier one.

Connections between uniform completions and the universal property for
order bounded operators were investigated in~\cite{Triki:02,Buskes:16}.

\section{$\Xru$ is the ``least'' uniformly complete vector lattice
  containing~$X$.}

We will now deduce from Theorem~\ref{univ} that $\Xru$ is, in a
certain sense, the least uniformly complete vector lattice
containing~$X$. Suppose that $X$ is a sublattice of a uniformly
complete vector lattice~$Z$. Let $J\colon X\hookrightarrow Z$ be the
inclusion map. By Theorem~\ref{univ}, it extends to an injective
lattice homomorphism $\widehat{J}\colon\Xru\to Z$. It follows that
$\Range\widehat{J}$ is a uniformly complete sublattice of $Z$
containing~$X$. Furthermore, if $X\subseteq Y\subseteq Z$ for some
uniformly complete vector lattice $Y$, we may view $J$ as a map from
$X$ to~$Y$. It is now easy to see that $\Range\widehat{J}$ is
contained in~$Y$. We can summarize this as follows:

\begin{theorem}\label{Xru-int-Z}
  Let $X$ be a sublattice of a uniformly complete vector
  lattice~$Z$. Then $\Xru$ is lattice isomorphic to the intersection
  of all uniformly complete sublattices of $Z$ containing~$X$;
  moreover, the isomorphism preserves~$X$.
\end{theorem}

Thus, in our definition of $\Xru$, one may replace $X^\delta$ with any
uniformly complete vector lattice containing~$X$. Combining this with
Proposition~\ref{u-compl-closed}, we get:

\begin{corollary}\label{ru-amb}
  If $X$ is a majorizing sublattice of a uniformly complete vector
  lattice $Y$ then $\Xru$ is lattice isomorphic to the uniform closure
  of $X$ in~$Y$.
\end{corollary}

Example~\ref{ex:c00} shows that the the majorizing assumption cannot be
dropped.

\begin{remark}\label{str-un-ru-compl}
  Let $X$ be a vector lattice with a strong unit~$e$. By
  Krein-Kakutani's representation theorem,
  $\bigl(X,\norm{\cdot}_e\bigr)$ embeds isometrically as a dense
  sublattice into a $C(K)$ space, with $e$ becoming~$\one$. Since
  $C(K)$ is uniformly complete and $X$ is uniformly dense and
  majorizing in $C(K)$, it follows from Corollary~\ref{ru-amb} that
  $C(K)$ is~$\Xru$. We conclude that $\Xru$ is the norm completion of
  $\bigl(X,\norm{\cdot}_e\bigr)$. It also follows that the uniform
  closure and the uniform adherence of $X$ in the norm completion
  agree.
\end{remark}

\begin{example}\label{loc-const}
  Let $Y=C[0,1]$; let $X$ be the set of all functions in $Y$ that are
  constant on $[0,\varepsilon)$ for some $\varepsilon>0$. Being a
  Banach lattice, $Y$ is uniformly complete. It is easy to see that
  $X$ is a majorizing sublattice, and that the uniform closure of $X$
  in $Y$ is all of~$Y$. Hence, we may identify $\Xru$ with $Y$. Note
  that $Y$ is not order complete.
\end{example}

\begin{example}\label{ex:piece-wise-ru}
  Let $X$ be the sublattice of $C[0,1]$ consisting of all piece-wise
  affine functions. Then $\Xru=C[0,1]$.
\end{example}

\begin{example}
  Let now $X$ be the space of all piece-wise affine
  functions in $C(\mathbb R)$; we will show that $\Xru$ is the space
  \begin{displaymath}
    Y=\Bigl\{f\in C(\mathbb R)\mid
    \lim\limits_{t\to\pm\infty}\frac{f(t)}{t}\mbox{ exists }\Bigr\}.
  \end{displaymath}
  It is clear that $X\subset Y$. Let
  \begin{displaymath}
    Z=\Bigl\{f\in C(\mathbb R)\mid
    \lim\limits_{t\to\pm\infty} f(t)\mbox{ exists }\Bigr\}.
  \end{displaymath}
  Clearly, $Z$ may be identified with $C(\overline{\mathbb R})$. Being
  a $C(K)$ space, $Z$ is an AM-space under
  $\norm{\cdot}_{\one}$. Define $u\in C(\mathbb R)$ and
  $T\colon Z\to Y$ via $u(t)=\abs{t}\vee 1$ and $Tf=uf$. It is easy to
  see that $T$ is a surjective lattice isomorphism. In particular,
  $T\one=u$ and $T$ is a $\norm{\cdot}_{\one}$-to-$\norm{\cdot}_u$
  isometry. It follows that $\bigl(Y,\norm{\cdot}_u\bigr)$ is a Banach
  lattice, hence it is uniformly complete.  Stone-Weierstrass Theorem
  yields that $T^{-1}(X)$ is $\norm{\cdot}_{\one}$-dense in $Z$,
  because we may view $T^{-1}(X)$ as a sublattice of
  $C(\overline{\mathbb R})$ which contains $\one$ and separates
  points.  Therefore, $X$ is $\norm{\cdot}_u$-dense in~$Y$. It now
  follows from Remark~\ref{str-un-ru-compl} that $\Xru=Y$.
\end{example}

\begin{example}
  It would be interesting to extend the preceding example to $\mathbb R^n$.
  It was shown in~\cite[Theorem~4.1]{Adeeb:03} that the space $LPA$ of
  all locally piece-wise affine functions is
  $\norm{\cdot}_{\one}$-dense in $C(\mathbb R^n)$; recall that a
  continuous real-valued function $f$ on $\mathbb R^n$ is said to be
  locally piece-wise affine if for every bounded convex subset
  $\Omega$ of $\mathbb R^n$ with non-empty interior there exist
  finitely many affine functions such that at every point of~$\Omega$,
  $f$ agrees with one of the affine functions. It follows from
  Corollary~\ref{ru-amb} that $LPA^{\rm ru}=C(\mathbb R^n)$.
\end{example}

\begin{example}
  The following example is motivated by the concepts of
  free vector lattices and free Banach lattices; see,
  e.g.,~\cite{dePagter:15,Taylor:24}.  Let $L_n$ be the space of all
  lattice linear functions in $C(\mathbb R^n)$, that is, functions
  that are obtained only using lattice and linear operations of the
  variables. Alternatively, $L_n$ is the sublattice of
  $C(\mathbb R^n)$ generated by the subspace of all linear
  functions. It is known that $L_n$ is (relatively) uniformly dense in
  the space $C_{\rm ph}(\mathbb R^n)$ of all continuous positively
  homogeneous real-valued functions on~$\mathbb R^n$. Indeed, the
  restriction map that sends $f\in C_{\rm ph}(\mathbb R^n)$ to its
  restriction to the unit sphere of $\ell^n_\infty$ is a lattice
  isomorphism between $C_{\rm ph}(\mathbb R^n)$ and
  $C(S_{\ell_\infty^n})$, and the image of $L_n$ under this
  restriction map is supremum-norm dense in $C(S_{\ell_\infty^n})$ by
  Stone-Weierstrass Theorem, hence uniformly dense. Hence, by
  Corollary~\ref{ru-amb}, $L_n^{\rm ru}=C_{\rm ph}(\mathbb R^n)$.

  One can identify $L_n$ and $C_{\rm ph}(\mathbb R^n)$ with $\FVL(n)$
  and $\FBL(n)$, respectively. So the preceding argument asserts that
  $\FVL(n)^{\rm ru}=\FBL(n)$. It is shown in~\cite{Emelyanov:24} that
  for an infinite $A$ one has $\FVL(A)^{\rm ru}\subsetneq\FBL(A)$.  It
  would be interesting to describe $\FVL(A)^{\rm ru}$ for an arbitrary
  set~$A$.
\end{example}

The following two statements are analogous to results in Section~2
of~\cite{Cernak:09}. It is straightforward that uniform convergence
passes down to projection bands.

\begin{proposition}\label{oplus-ru}
  $(X\oplus Y)^{\rm ru}=\Xru\oplus\Yru$
  for any two vector lattices $X$ and~$Y$. Furthermore,
  $\Xru\oplus 0$ is the uniform closure and the ideal generated by
  $X\oplus 0$ in $\Xru\oplus\Yru$.
\end{proposition}

\begin{proof}
  Since $\Xru$ and $\Yru$ are uniformly complete, using the definition
  of uniform completeness via uniformly Cauchy nets one can show that
  $\Xru\oplus\Yru$ is uniformly complete. The ``furthermore'' clause
  follows from Corollary~\ref{ru-clos-ru} and the facts that $X$ is
  majorizing in $\Xru$ and $\Xru\oplus 0$ is a projection band in
  $\Xru\oplus\Yru$.

  Let $T\colon X\oplus Y\to\Xru\oplus\Yru$ be the canonical inclusion
  map. Let $\widehat{T}\colon(X\oplus Y)^{\rm ru}\to\Xru\oplus\Yru$ be
  its extension as in Theorem~\ref{univ}. Then $\widehat{T}$ is an
  injective lattice homomorphism. It is left to show that it is
  surjective. Since $\Range\widehat{T}$ is lattice isomorphic to
  $(X\oplus Y)^{\rm ru}$, it is uniformly complete. Since it
  contains $\Range T=X\oplus Y$, $\Range\widehat{T}$ is majorizing in
  $\Xru\oplus\Yru$. By Theorem~\ref{u-compl-closed}, it is uniformly
  closed. Since it contains $X\oplus 0$, it contains its closure
  $\Xru\oplus 0$. Similarly, $\Range\widehat{T}$ contains
  $0\oplus\Yru$. It follows that $\Range\widehat{T}=\Xru\oplus\Yru$.
\end{proof}



\begin{corollary}\label{proj-band-ru}
  If $Y$ is a projection band in $X$ then $\Yru$ equals the uniform
  closure of $Y$ in~$\Xru$; it also equals the ideal generated by $Y$
  in~$\Xru$. Furthermore, $\Yru$ is a projection band in~$\Xru$.
\end{corollary}



In many categories, the completion of an object is a complete object
that contains the original objects as a dense sub-object. In
particular, if $X$ is a topological vector space, by its completion we
mean a complete topological vector space $Y$ which contains $X$ as a
linear and topological subspace which is dense in
$Y$. Part~\eqref{cl-compl-pass} of the next result may be viewed as an
analogue of this for uniform completeness of vector lattices.

\begin{proposition}\label{cl-compl}
  Let $Y$ be a uniformly complete vector lattice
  and $X$ a sublattice of $Y$ such that $\ucl{X}=Y$. TFAE:
  \begin{enumerate}
  \item\label{cl-compl-ru} $Y$ is isomorphic to $\Xru$ via an
    isomorphism preserving~$X$;
  \item\label{cl-compl-maj} $X$ is majorizing in~$Y$;
  \item\label{cl-compl-pass} Uniform convergence passes down from $Y$
    to~$X$.
  \end{enumerate}
\end{proposition}

\begin{proof}
  \eqref{cl-compl-ru}$\Rightarrow$\eqref{cl-compl-maj} by
  Proposition~\ref{ru-maj}.
  \eqref{cl-compl-maj}$\Rightarrow$\eqref{cl-compl-ru} by
  Corollary~\ref{ru-amb}.
  \eqref{cl-compl-maj}$\Rightarrow$\eqref{cl-compl-pass} by
  Proposition~\ref{majorizing}.

  \eqref{cl-compl-pass}$\Rightarrow$\eqref{cl-compl-ru} By
  Theorem~\ref{Xru-int-Z} we may assume that $\Xru\subseteq Y$. We
  claim that every subset of $\Xru$ which is order bounded in $Y$ is
  also order bounded in $\Xru$.  Let $A\subseteq\Xru_+$ and $y\in Y$
  such that $A\le y$. By Proposition~\ref{pass-down},
  $[0,y]\cap X\le x$ for some $x\in X_+$. Since $X$ is order dense in
  $\Xru$, for every $z\in A$ we have $z=\sup[0,z]\cap X$ in $\Xru$. It
  follows from $[0,z]\cap X\subseteq[0,y]\cap X$ that $z\le x$. This
  yields $A\le x$. By Proposition~\ref{pass-down}, uniform convergence
  passes from $Y$ to $\Xru$. It follows that $\Xru$ is uniformly
  closed in~$Y$, we conclude that $Y=\Xru$.
\end{proof}

\begin{example}\label{ex:pos-inj}
  The following example shows that Theorem~\ref{univ} fails for the
  class of positive injective order continuous operators.  Let $K$ be
  the Cantor set, viewed as a subset of $[0,1]$. Put $Y=C(K)$. Being a
  Banach space, $Y$ is uniformly complete. Let $Z$ be the subspace of
  $Y$ spanned by the characteristic functions of clopen sets. It is
  dense (in norm and uniformly) by Stone-Weierstrass Theorem. Put
  $X=Z\oplus Z$. It follows from Remark~\ref{str-un-ru-compl} or from
  Proposition~\ref{oplus-ru} that $\Xru=Y\oplus Y$. Let
  $h\colon K\to\mathbb R$ be the inclusion map, i.e., $h(t)=t$. Define
  $T\colon \Xru\to Y$ via $T(f\oplus g)=f+gh$, that is,
  $\bigl(T(f\oplus g)\bigr)(t)=f(t)+tg(t)$ for all $t\in K$. Clearly,
  $T\ge 0$. It can be easily verified that~$T$, as well as its
  restriction to~$X$, are order continuous. Note that $T$ is not
  injective as $T(-h\oplus\one)=0$.

  We claim that $T_{|X}$ is injective. Indeed, suppose that $T(f\oplus
  g)=0$ for some $f,g\in Z$. We can write them as
  \begin{math}
    f=\sum_{i=1}^na_i\one_{K_i}
  \end{math}
  and
  \begin{math}
    g=\sum_{i=1}^nb_i\one_{K_i},
  \end{math}
  where
  $K_1,\dots,K_n$ are disjoint clopen non-empty subsets of~$K$. Fix
  $m=1,\dots,n$. Since $K$ has no isolated points, there are distinct
  points $s$ and $t$ in~$K_m$. We have
  \begin{displaymath}
    a_m+sb_m=f(s)+h(s)g(s)=0=f(t)+h(t)g(t)=a_m+tb_m.
  \end{displaymath}
  It follows that $a_m=b_m=0$ for all $m$, so that $f=g=0$.
\end{example}

In Theorem~\ref{univ}, we established the universal property for
operators from certain classes. Let $X$ and $Y$ be two vector lattices
and $T\colon X\to Y$ be an operator from one of these classes. Let
$j_X\colon X\hookrightarrow\Xru$ and $j_Y\colon Y\hookrightarrow\Yru$
be the canonical embeddings. Since $j_YT\colon X\to\Yru$ is again in
the same class as $T$, by the theorem it extends uniquely to an
operator $\Tru\colon\Xru\to\Yru$ and, moreover, $\Tru$ is again an
operator from the same class:
\begin{displaymath}
     \xymatrix{
       X \ar@{->}[r]^{T} \ar@{->}[d]_{j_X} &
       Y \ar@{->}[d]^{j_Y}  \\
      \Xru \ar@{->}[r]^{\Tru} & \Yru}
\end{displaymath}

We will next present a counterexample to the following two natural conjectures:
\begin{itemize}
  \item If $T\colon X\to Y$ is interval preserving, then $\Tru$ is interval
    preserving.
  \item ru-completion operation commutes with taking principal
    ideals. That is, for $f\in Y$, the principal ideal of $f$ in
    $\Yru$ is lattice isomorphic to the ru-completion of the principal
    ideal of $f$ in~$Y$. In symbols,
    $I_{\Yru}(f)\simeq\bigl(I_Y(f)\bigr)^{\mathrm{ru}}$.
\end{itemize}

\begin{example}\label{ex:piece-wise01}(cf.~\cite{Azouzi})
  Let $Y$ be the space of all piece-wise affine functions on
  $[0,1]$. As in Example~\ref{ex:piece-wise-ru}, $\Yru=C[0,1]$.
  Take $f\in Y$ given by $f(t)=t$. Let $X=I_Y(f)$, and let $T\colon
  X\to Y$ be the inclusion map. It is clear that $T$ is an interval
  preserving lattice isomorphic embedding.

  We claim that $\Xru$ is lattice isomorphic to $C[0,1]$. Indeed,
  consider the operator $S\colon X\to C[0,1]$ defined by $Sg=g/f$ for
  $g\in X$. Note that $g/f$ is clearly defined and continuous on
  $(0,1]$ and is constant on $(0,\varepsilon]$ for some
  $\varepsilon>0$, so $g/f$ extends by continuity to a function in
  $C[0,1]$; it is this function that we take for~$Sg$. It is easy to
  see that $S$ is a lattice isomorphic embedding. Hence its range $SX$
  is a sublattice of $C[0,1]$. It is also easy to see that $SX$ contains
  constant functions and separates points of $[0,1]$. By
  Stone-Weierstrass Theorem, $SX$ is norm dense in $C[0,1]$. Hence,
  $X$ is lattice isomorphic to a uniformly dense majorizing sublattice
  of $C[0,1]$. Corollary~\ref{ru-amb} yields that $\Xru\simeq C[0,1]$;
  this proves the claim.

  On one hand, it follows from $I_Y(f)=X$ that
  \begin{math}
    \bigl(I_Y(f)\bigr)^{\mathrm{ru}}=\Xru\simeq C[0,1].
  \end{math}
  On the other hand,
  \begin{math}
    I_{\Yru}(f)=I_{C[0,1]}(f)\simeq C_b(0,1]
  \end{math}
  as in Example~\ref{COmega-ucompl}. It is easy to see that
  $C[0,1]\not\simeq C_b(0,1]$. Indeed, while $C[0,1]$ is separable,
  $C_b(0,1]$ contains $\ell_\infty$ as a closed sublattice and,
  therefore, is non-separable. However, since both spaces are Banach
  lattices, every surjective lattice isomorphism between them must
  also be a norm isomorphism, hence must preserve separability. 
  We conclude that
  $I_{\Yru}(f)\not\simeq\bigl(I_Y(f)\bigr)^{\mathrm{ru}}$.

  Finally, we claim that $\Tru\colon\Xru\to\Yru$ fails to be
  interval preserving. Since $T$ is a lattice isomorphic embedding, so
  is~$\Tru$. Suppose that $\Tru$ is interval preserving. Then it maps
  principal ideals to principal ideals, hence
  \begin{math}
    \Tru(I_{\Xru}(f))=I_{\Yru}(\Tru f).
  \end{math}
  However, since
  $X$ is majorizing in~$\Xru$,  we have
  \begin{math}
    \Tru(I_{\Xru}(f))=\Tru\Xru\simeq\Xru\simeq C[0,1],
  \end{math}
  while
  \begin{math}
    I_{\Yru}(\Tru f)\simeq C_b(0,1],
  \end{math}
  which is a contradiction.
\end{example}

\section{The adherence of $X$ with regulators in $X$.}

The following is a simple consequence of Corollary~\ref{ru-amb}.

\begin{proposition}\label{ru-IX}
  If $X$ is a sublattice of a uniformly complete vector lattice $Z$
  then $\Xru$ is lattice isomorphic to the uniform closure of $X$ in
  the ideal $I(X)$ generated by $X$ in~$Z$.
\end{proposition}

Let $\Xruu$ denote the uniform adherence of $X$ in~$\Xru$.
If $X$ is a sublattice of a vector lattice $Z$ then, by
Proposition~\ref{loc-adh}, the uniform adherence of $X$ in $I(X)$
is the set of all uniform limits in $Z$ of sequences in $X$ with
regulators in~$X$.

\begin{proposition}
  Suppose that $X$ is a sublattice of a uniformly complete vector
  lattice~$Z$. Then the uniform adherence of $X$ in $I(X)$ is
  isomorphic~$\Xruu$.
\end{proposition}

\begin{proof}
  By Proposition~\ref{ru-IX}, the uniform closure of $X$ in $I(X)$ is
  isomorphic~$\Xru$. Hence, if $z$ is in the uniform adherence of $X$
  in $I(X)$, it is in~$\Xru$, and there is a sequence in $X$ which
  converges to $z$ with respect to a regulator in~$X$. It follows that
  $z\in\Xruu$. The converse is straightforward.
\end{proof}

The preceding result also implies that the uniform adherence of $X$ in
$I(X)$ does not depend on the ambient space~$Z$. The following is an
immediate consequence of Corollary~\ref{IuX-cofin} and
Proposition~\ref{ideals-up}.

\begin{corollary}
  Let $X$ be a sublattice of a uniformly complete vector
  lattice~$Z$. If $\Icl{u}{X}$ is an ideal in $\Icl{v}{X}$ whenever
  $u\le v$ in $X_+$ then $\Xruu=\bigcup_{u\in X_+}\Icl{u}{X}=\Xru$.
\end{corollary}

\medskip

The following approach to constructing $\Xru$ was outlined
in~\cite{Veksler:69} and then further developed in
\cite{Emelyanov:23,Emelyanov:24}. Suppose that $X$ is a sublattice of
a uniformly complete vector lattice~$Z$. Consider the set of all
uniform limits in $Z$ of sequences in $X$ with regulators in~$X$. By
the preceding argument, this set is precisely the uniform adherence of
$X$ in $I(X)$; it may also be identified with $\Xruu$, the uniform
adherence of $X$ in~$\Xru$. Iterating this process, for every ordinal
number $\kappa$ we get $\uadh{X}{\kappa}$ in~$\Xru$. As observed in
the introduction, $\uadh{X}{\omega_1}=\Xru$, so the process stabilizes
at $\Xru$ after $\omega_{1}$ steps.

\begin{remark}
  Another approach is developed in \cite[Lemma~1.1]{Buskes:89}.  They
  introduce a topology on $Z$ where the sets are closed if and only if
  they are closed with respect to the uniform convergence with
  regulators in~$X$. They prove that the closure of $X$ in $Z$ with
  respect to this topology is a uniformly complete sublattice which
  essentially does not depend on~$Z$. Moreover, $X$ is majorizing in
  this closure, and the universal property for lattice homomorphisms
  is satisfied for the closure.
\end{remark}

Yet another related construction was discussed in~\cite{Cernak:09} in
the setting of lattice ordered groups. One can construct the norm
completion of a normed spaces as the quotient of the space of all norm
Cauchy sequences over all norm null sequences.  It is a natural
question whether one could construct $\Xru$ in a similar fashion. The
following propositions says that this construction only yields
$\Xruu$ rather than all of~$\Xru$.

Let $X^{\mathbb N}$ be the set of all sequences of elements of~$X$. It
is a standard fact that $X^{\mathbb N}$ is again a vector
lattice under coordinate-wise order and operations. Let $C(X)$ and
$N(X)$ be the subsets of $X^{\mathbb N}$ consisting of all uniformly
Cauchy and uniformly null sequences, respectively. It is easy to
verify that $C(X)$ is a sublattice of $X^{\mathbb N}$ while $N(X)$ is
an ideal of $C(X)$. It follows that $C(X)/N(X)$ is a vector lattice.

\begin{proposition}
  In the notations above, $C(X)/N(X)$ is lattice isomorphic to~$\Xruu$.
\end{proposition}

\begin{proof}
  Let $\xi\in C(X)/N(X)$. Take a representative $(x_n)$
  in~$\xi$. Being a uniformly Cauchy sequence, $(x_n)$ converges to
  some $x$ in~$\Xru$. Since $X$ is majorizing in~$\Xru$, we may choose
  the regulator in~$X$, hence $x\in\Xruu$. It is easy to see that $x$
  does not depend on a particular choice of the representative. Put
  $T\xi=x$. This defines $T\colon C(X)/N(X)\to\Xruu$.  Clearly, $T$ is
  linear; it is a lattice homomorphism because $x_n\goesu x$ implies
  $\abs{x_n}\goesu\abs{x}$. Observe that $T$ is one-to-one: if $x=0$
  then $(x_n)\in N(X)$, hence $\xi=0$.  It follows from the definition
  of $\Xruu$ that $T$ is surjective.
\end{proof}

\section{When $\uadh{X}{1}=\ucl{X}$ in $X^\delta$?}

Throughout this section, $X$ will be a vector lattice; we
will write $\uadh{X}{1}$ and $\ucl{X}$ for the adherence and the
closure of $X$ in~$X^\delta$. Since $X$ is majorizing in $X^\delta$,
we have $\uadh{X}{1}=\Xruu$ in this setting. We know from
Corollary~\ref{ru-clos-ru} that $\ucl{X}=\Xru$. In this section, we
will provide several sufficient conditions when $\uadh{X}{1}=\ucl{X}$.

Recall that a vector lattice satisfies the \term{$\sigma$-property} if every
countable set is contained in a principal ideal.

\begin{proposition}[\cite{Quinn:75}]
  If $X$ has the $\sigma$-property then $\uadh{X}{1}=\ucl{X}$.
\end{proposition}

\begin{proof}
  It suffices to prove that $\uadh{X}{1}=\uadh{X}{2}$. Suppose that
  $z\in \uadh{X}{2}$. We can find a sequence $(y_n)$ in $\uadh{X}{1}$
  such that $y_n\goesu x$ in~$X^\delta$. For each $n$, we can find a
  sequence $(x^{(n)}_k)$ in $X$ such that $x^{(n)}\goesu y_n$. Using
  the fact that $X$ is majorizing in $X^\delta$ and that $X$ and,
  therefore, $X^\delta$ has the $\sigma$-property, we find a $u\in X$ such
  that the vectors $z$, $y_n$, and $(x^{(n)}_k)$ (for all $n$ and $k$) are
  in~$I_u^{X^\delta}$, the principal ideal of $u$
  in~$X^\delta$. Moreover, we may also assume that $y_n\goesur{u}z$
  and $x^{(n)}\goesur{u} y_n$ for every $n$. We may identify
  $I_u^{X^\delta}$ with a $C(K)$ space. It follows that $z$ is in the
  second norm adherence of the set $\{x^{(n)}_k\}_{n,k}$. Since norm
  convergence is topological, $z$ is in the first norm adherence of
  $\{x^{(n)}_k\}_{n,k}$. We conclude that $z$ is in~$\uadh{X}{1}$.
\end{proof}

Recall that a net $(x_\alpha)$ in a vector lattice
\term{$\sigma$-order converges} to a vector $x$ if there exists a
sequence $(z_n)$ in $X$ such that $z_n\downarrow 0$ and for every $n$
there exists $\alpha_0$ such that $\abs{x_\alpha-x}\le z_n$ whenever
$\alpha\ge\alpha_0$. We write $x_\alpha\goesso x$. Lattice operations
are $\sigma$-order continuous; see, e.g.,~\cite{BCTW}. It is easy to
see that $x_\alpha\goesu x$ implies $x_\alpha\goesso x$.

The following lemma is a special case of Theorem~5.2 of~\cite{BCTW}.

\begin{lemma}\label{so-ru-compl}
  If $X$ is complete with respect to $\sigma$-order convergence then
  it is uniformly complete.
\end{lemma}

\begin{proof}
  Suppose that $(x_\alpha)$ is uniformly Cauchy. Then it is
  $\sigma$-order Cauchy, hence $x_\alpha\goesso x$ for some~$x$.
  Since $(x_\alpha)$ is uniformly Cauchy, we have
  $x_\alpha-x_\beta\goesur{u}0$ for some regulator $u\in X_+$. Fix
  $\varepsilon>0$. Find $\alpha_0$ such that
  $\abs{x_\alpha-x_\beta}\le\varepsilon u$ whenever
  $\alpha,\beta\ge\alpha_0$. Passing to the $\sigma$-order limit on
  $x_\beta$, we get $\abs{x_\alpha-x}\le\varepsilon u$ whenever
  $\alpha\ge\alpha_0$. This implies that $x_\alpha-x\goesur{u}0$.
\end{proof}

Note that if $Y$ is a sublattice of $X^\delta$ containing $X$ then $Y$
is order dense and, therefore, regular in~$X^\delta$.

\begin{lemma}\label{msod-so-ru}
  Let $Y$ be a majorizing super order dense sublattice of~$X$.
  Suppose that $\sigma$-order convergence and uniform convergence
  agree on sequences in~$Y$. Then they also agree on nets in~$X$.
\end{lemma}

\begin{proof}
  Suppose that $x_\alpha\goesso 0$ in~$X$. Find a sequence $(z_n)$ in
  $X$ such that $z_n\downarrow 0$ and for every $n$ there exists
  $\alpha_0$ such that $\abs{x_\alpha}\le z_n$ whenever
  $\alpha\ge\alpha_0$. For every~$n$,
  there is a countable set
  $A_n\subseteq Y$ such that $z_n=\inf A_n$ in~$X$.
  Let $A=\bigcup_{n=1}^\infty A_n$, and let $v_m$ be an
  enumeration of $A$. Put $u_m=\bigwedge_{i=1}^mv_m$. Clearly,
  $u_m\downarrow$. We have
  \begin{displaymath}
    \inf u_m=\inf v_m=\inf A=\inf_n\inf A_n=\inf_nz_n=0.
  \end{displaymath}
  It follows that $u_m\goesso 0$ in~$Y$. By assumption,
  $u_m\goesur{v} 0$ in~$Y$, with some regulator $v\in Y_+$. It is easy to
  see that for every $m$ there exists $n$ such that $z_n\le u_m$; it
  follows that there exists $\alpha_0$ such that
  $\abs{x_\alpha}\le u_m$ whenever $\alpha\ge\alpha_0$. We conclude
  that $x_\alpha\goesur{v} 0$.
\end{proof}

The following theorem is analogous to Theorem~6.7 in~\cite{Quinn:75}.

\begin{theorem}\label{so-ru-scompl}
  Suppose that $\sigma$-order convergence and uniform convergence
  agree on sequences in $X$. Then they also agree on nets
  in~$\Xru$. Furthermore, $\uadh{X}{1}=\ucl{X}$, and $\ucl{X}$ is the
  intersection of all sublattices of $X^\delta$ containing $X$ and
  complete under $\sigma$-order convergence.
\end{theorem}

\begin{proof}
  By Proposition~\ref{ru-maj} and Lemma~\ref{msod-so-ru},
  $\sigma$-order convergence and uniform convergence agree on nets
  in~$\Xru$.

  It follows from Corollary~\ref{ru-clos-ru} and our definition of
  $\Xru$ that $\ucl{X}$ is the intersection of all uniformly complete
  sublattices $Y$ of $X^\delta$ containing~$X$. By
  Lemma~\ref{so-ru-compl}, $\ucl{X}$ is contained in the intersection
  of all sublattices of $X^\delta$ containing $X$ and complete under
  $\sigma$-order convergence. In fact, $\ucl{X}$ equals this
  intersection because $\ucl{X}$ is itself complete under
  $\sigma$-order convergence since it is uniformly complete and
  $\sigma$-order convergence and uniform convergence agree on it.

  To prove that $\uadh{X}{1}=\ucl{X}$, let $x\in\ucl{X}_+$. By
  Proposition~\ref{maj-o-dense}, there exists a sequence $(x_n)$ in
  $X_+$ such that $x_n\uparrow x$. It follows that $x_n\goesso x$ and,
  therefore, $x_n\goesu x$, so that $x\in\uadh{X}{1}$.
\end{proof}

For the proof of the next theorem, we need a few auxiliary facts and
definitions. We say that $X$ is has the \term{$\sigma$-projection
  property} ($\sigma$-PP) if every band generated by a countable set
is a projection band.

\begin{lemma}\label{sPP-ideals}
  The $\sigma$-projection property is inherited by ideals.
\end{lemma}

\begin{proof}
  Suppose that $X$ has the $\sigma$-PP, $J$ is an ideal in~$X$, and
  $A$ is a countable subset of~$J$. Note that $B_J(A)$, the band
  generated by $A$ in $J$, agrees with $B_X(A)\cap J$. It is now easy
  to see that this is a projection band in~$J$.
\end{proof}

Recall that a topological space is \term{basically disconnected} if
the closure of every open $F_\sigma$ set is open.  The following is
Theorem~3.14 of~\cite{Bilokopytov:24}:

\begin{proposition}\label{K-bas-sPP}
  Let $K$ be a compact Hausdorff space. The space $C(K)$ has a dense
  sublattice with the $\sigma$-projection property iff $K$ is basically
  disconnected.
\end{proposition}

The following fact is Lemma~7.25 in~\cite{Aliprantis:03}. The
statement of the lemma in~\cite{Aliprantis:03} requires that $\Omega$
be extremally  disconnected, but it is easy to
see that the proof remains valid when $\Omega$ is just basically
disconnected.

\begin{lemma}
  Let $\Omega$ be a basically disconnected topological space~$\Omega$,
  $U$ an open subset of $\Omega$ and $f\in C(U)$. Then $f$ extends
  uniquely to a continuous function from $\overline{U}$
  to~$\overline{\mathbb R}$.
\end{lemma}

\begin{example}\label{COmega-Iu-bas}
  Suppose that the topological space $\Omega$ in
  Example~\ref{COmega-ucompl} is basically disconnected. The
  set~$\Omega_0$, constructed in the example, is an open
  $F_\sigma$-set, hence $\overline{\Omega_0}$ is clopen; in
  particular, it is also basically disconnected. By the preceding
  lemma, the function $h$ in the example extends uniquely to a
  function in $C(\overline{\Omega_0})$, Hence, $I_u$ is lattice
  isometric to $C(\overline{\Omega_0})$.
\end{example}

\begin{lemma}\label{bas-IuY-ICK}
  Let $K$ be a basically disconnected compact Hausdorff space, $Y$
  a norm dense sublattice of $C(K)$, and $u\in Y_+$. Then $I_u^Y$ is
  dense in $\bigl(I^{C(K)}_u,\norm{\cdot}_u\bigr)$.
\end{lemma}

\begin{proof}
  Let $L=\overline{\{u\ne 0\}}$. As in Example~\ref{COmega-Iu-bas}, we
  can construct a surjective lattice isometry
  \begin{math}
    J\colon C(L)\to\bigl(I^{C(K)}_u,\norm{\cdot}_u\bigr)
  \end{math}
  such that $(Jf)(s)=u(s)f(s)$ whenever $f\in C(L)$ and $s\in L$.
  It suffices to show that $J^{-1}(I_u^Y)$ is dense in $C(L)$. We will
  deduce this from Stone-Weierstrass Theorem. Clearly,
  $\one_L=J^{-1}u\in J^{-1}(I_u^Y)$.

  Suppose $s\ne t$ in~$L$. By a variant of Urysohn's lemma for dense
  sublattices of $C(K)$ as in Proposition~2.1 in~\cite{Bilokopytov:24},
  we can find $v\in Y$ such that $0\le v\le u$, $v$ vanishes on some
  neighborhood $U$ of $s$ in $K$ and agrees with $u$ on some
  neighborhood $V$ of $t$ in~$K$. Since $s\in L$, we can find a net
  $(s_\alpha)$ in $\{u\ne 0\}$ such that $s_\alpha\to s$. Then
  $s_\alpha\in U$ for all sufficiently large $\alpha$, so that
  \begin{math}
    (J^{-1}v)(s_\alpha)=\frac{v(s_\alpha)}{u(s_\alpha)}=0.
  \end{math}
  Passing to the limit, we get $(J^{-1}v)(s)=0$. Arguing similarly, we
  get $(J^{-1}v)(t)=1$. Hence, by Stone-Weierstrass Theorem, $J^{-1}(I_u^Y)$
  is dense in $C(L)$
\end{proof}

We say that a vector lattice is \term{countably order complete} if
every countable set which is bounded above has supremum. In the
literature, this concept is often called ``$\sigma$-order complete'',
but we would rather avoid the latter name as it can be confused with
completeness with respect to $\sigma$-order convergence.
Recall that a compact Hausdorff space $K$ is basically disconnected
iff $C(K)$ is countably order complete.

The following result is analogous to Theorem~8.6 in~\cite{Quinn:75},
see also~\cite{Bondarev:74}.

\begin{theorem}\label{sPP-clos}
  Suppose that $X$ has the $\sigma$-PP. Then
  $\uadh{X}{1}=\ucl{X}$. Moreover, this set is
  the least countably order complete sublattice of
  $X^\delta$ containing~$X$.
\end{theorem}

\begin{proof}
  It is clear that $\uadh{X}{1}\subseteq\ucl{X}$. We will prove that
  $\ucl{X}\subseteq\uadh{X}{1}$.

  Fix $u\in X_+$. We can view the ideal $I_u^X$ as a subset of the
  ideal~$I_u^{X^\delta}$. We will denote by $J_u$ the closure of
  $I_u^X$ in $I_u^{X^\delta}$ under~$\norm{\cdot}_u$. Since $X^\delta$
  is uniformly complete, the space
  $\bigl(I_u^{X^\delta},\norm{\cdot}_u\bigr)$ is complete. It follows
  that $J_u$ is the norm completion of
  $\bigl(I_u^X,\norm{\cdot}\bigr)$.  By Krein-Kakutani's
  representation theorem, we can represent $J_u$ as $C(K)$ for some
  compact Hausdorff space $K$. By Lemma~\ref{sPP-ideals},
  $I_u^X$ has the $\sigma$-PP. By
  Proposition~\ref{K-bas-sPP}, $K$ is basically disconnected, hence
  $C(K)$ is countably order complete. It follows that $J_u$ is
  countably order complete.

  Let $x\in X$ with $0\le x\le u$. Applying Lemma~\ref{bas-IuY-ICK}
  with $Y=I_u^X$ and observing that $I_x^Y=I_x^X$, we conclude that
  $I_x^X$ is dense in $\bigl(I_x^{C(K)},\norm{\cdot}_x\bigr)$.
  It follows that $J_x$ can be identified with $I_x^{C(K)}$. Since
  $C(K)=J_u$, it follows that $J_x$ is an ideal in~$J_u$.

  We have proved that $J_u$ is countably order complete and $J_x$ is
  an ideal in $J_u$ whenever $0\le x\le u$.
  By Corollary~\ref{IuX-cofin} and
  Proposition~\ref{ideals-up} we have that that
  $\uadh{X}{1}=\bigcup_{x\in X_+}J_x=\ucl{X}$ and $J_u$ is the
  principal ideal generated by $u$ in~$\uadh{X}{1}$.

  It is now easy to see that $\uadh{X}{1}$ is countably order
  complete. Indeed, if $C$ is a countable subset of $\uadh{X}{1}_+$
  and $C\le w$ for some $w\in \uadh{X}{1}_+$, then we can find
  $x\in X$ with $w\le x$, then $C\le x$ in $J_x$, hence $\sup C$
  exists in $J_x$ and, therefore, in~$\uadh{X}{1}_+$, because $J_x$ is
  regular in~$\uadh{X}{1}$.

  It is a standard fact that every countably order complete vector
  lattice is uniformly complete. If follows from the definition of
  $\Xru$ that $\Xru\subseteq\uadh{X}{1}$. We conclude that
  $\Xru=\uadh{X}{1}$, and this is the least countably order
  complete sublattice of $X^\delta$ containing~$X$.
\end{proof}

The following theorem is implicitly contained in~\cite{Veksler:69}.

\begin{theorem}\label{PP-clos}
  Suppose that $X$ has the PP. Then
  $\uadh{X}{1}=\ucl{X}=X^\delta$.
\end{theorem}

\begin{proof}
  The proof is similar except that we use Theorem~3.12
  in~\cite{Bilokopytov:24} instead of Theorem~3.14: $C(K)$ has a dense
  sublattice with the PP iff $K$ is extremally disconnected
  iff $C(K)$ is order complete. The proof then yields that
  $\uadh{X}{1}$ is order complete. It follows that
  $\uadh{X}{1}=X^\delta$.
\end{proof}


It is well known that if a
sublattice $X$ of a vector lattice $Y$ is order dense and
order complete then it is an ideal in~$Y$; see, e.g., Theorem~1.40
in~\cite{Aliprantis:03}. A similar argument shows that if $X$ is super
order dense in $Y$ and countably order complete then it is an ideal in~$Y$.

\begin{corollary}\label{CK-dense}
  Let $X$ be a sublattice of some $C(K)$ space such
  that $\one\in X$. Each of the following conditions implies that $X$
  is norm dense in $C(K)$:
  \begin{enumerate}
  \item\label{CK-dense-PP} $X$ is order dense and has the PP;
  \item\label{CK-dense-sPP} $X$ is super order dense and has
    the $\sigma$-PP;
  \item\label{CK-dense-so} $X$ is super order dense and $\sigma$-order
    convergence and uniform convergence agree on sequences in~$X$.
  \end{enumerate}
\end{corollary}

\begin{proof}
  By Theorem~\ref{Xru-int-Z}, we may view $\Xru$ as a sublattice of
  $C(K)$. Since $X$ is order dense and majorizing in $C(K)$, the same
  is true for~$\Xru$; it follows that $\Xru$ is regular in $C(K)$. By
  Corollary~\ref{ru-amb}, $\Xru$ equals the relative uniform closure
  of $X$ in $C(K)$. Note that since relative uniform convergence on
  $C(K)$ agrees with norm convergence, this means that $X$ is norm
  dense in~$\Xru$. It suffices to show $\Xru=C(K)$.

  \eqref{CK-dense-PP}
  Note that $\Xru$ is order dense in $C(K)$; also, $\Xru$ is order complete by
  Theorem~\ref{PP-clos}. It follows that $\Xru$ is an ideal in
  $C(K)$. Since $\one\in\Xru$, we have $\Xru=C(K)$.

  \eqref{CK-dense-sPP}
  is proved similarly using Theorem~\ref{sPP-clos} and the remark
  before the corollary.

  \eqref{CK-dense-so}
  Let $f\in C(K)_+$. Find a sequence $(x_n)$ in
  $X$ such that $x_n\uparrow f$ in $C(K)$. In particular,
  $x_n\goesso f$ in $C(K)$. By Lemma~\ref{msod-so-ru}, $x_n\goesu f$ in
  $C(K)$, hence $f\in\ucl{X}=\Xru$.
\end{proof}

\begin{example}\label{ex:finL0}
  Let $(\Omega,\mu)$ be a semi-finite measure space, i.e., every set
  of positive measure contains a subset of finite positive
  measure. Let $X$ be the span of all characteristic functions of
  measurable sets of finite measure in $L_0(\mu)$. It is easy to see
  that $X$ is an order dense sublattice of $L_0(\mu)$, and that it has
  the PP. According to Theorem~\ref{PP-clos}, $\Xru=X^\delta$. It is
  well known that $X^\delta$ is the ideal generated by $X$ in its
  universal completion. Since the universal completion of $X$ is
  $L_0(\mu)$, we conclude that $\Xru$ is the space of all essentially
  bounded functions in $L_0(\mu)$ with support of finite measure.
\end{example}

\begin{question}[\cite{Quinn:75}]
  Assume that $\Xru=X^\delta$ or even $\uadh{X}{1}=X^\delta$. Does
  this imply that $X$ has the PP? Let $X^\sigma$ be the
  intersection of all countably order complete sublattices of
  $X^\delta$ containing $X$. By Theorem~\ref{sPP-clos}, if $X$ has the
  $\sigma$-PP then $\Xru=\uadh{X}{1}=X^\sigma$. Is the
  converse true?  That is, does the $\sigma$-PP follow from
  $\Xru=X^\sigma$ or $\uadh{X}{1}=X^\sigma$?
\end{question}

Let now $X$ be a normed lattice and $\widetilde{X}$ its norm
completion. Being a Banach lattice, $\widetilde{X}$ is uniformly
complete. It follows from Theorem~\ref{Xru-int-Z} that we may view
$\Xru$ as a sublattice of $\widetilde{X}$. Example~\ref{ex:c00} shows
that, in general, the inclusion may be proper. The following theorem
(cf.~\cite{Danilenko:81}) characterizes when the two spaces are equal.

\begin{theorem}\label{ncompl}
  Let $X$ be a normed lattice and $\widetilde{X}$ its norm
  completion. TFAE:
\begin{enumerate}
  \item\label{ncompl-maj} $X$ is majorizing in~$\widetilde{X}$;
  \item\label{ncompl-obdd} Every norm Cauchy sequence in $X$ has an
    order bounded subsequence;
  \item\label{ncompl-inc} Every increasing norm Cauchy sequence in $X$ is
    order bounded;
  \item\label{ncompl-Cauchy} Every norm Cauchy sequence in $X$ has a
    uniformly Cauchy subsequence;
  \item\label{ncompl-1X} Every element of $\widetilde{X}$ is a uniform
    limit in $\widetilde{X}$ of a sequence in $X$ with a regulator in~$X$;
  \item\label{ncompl-ru} $\Xru=\widetilde{X}$.
  \end{enumerate}
\end{theorem}

\begin{proof}
  \eqref{ncompl-maj}$\Rightarrow$\eqref{ncompl-obdd} A norm Cauchy
  sequence in $X$ is norm convergent in~$\widetilde{X}$, hence it has a
  subsequence which is order bounded in $\widetilde{X}$ and, therefore,
  in~$X$.

  The implication \eqref{ncompl-obdd}$\Rightarrow$\eqref{ncompl-inc}
  is straightforward.

  \eqref{ncompl-inc}$\Rightarrow$\eqref{ncompl-Cauchy}
  Suppose that $(x_n)$ is norm Cauchy in $X$. Passing to a subsequence, we
  may assume that $\norm{x_n-x_{n+1}}\le\frac{1}{4^n}$. For each~$m$,
  put
  \begin{math}
    y_m\coloneqq\sum_{k=1}^m2^k\abs{x_{k+1}-x_k}.
  \end{math}
  Clearly, $(y_m)$ is an increasing sequence in $X$; it is also easy
  to see that if $n\ge m$ then
  \begin{displaymath}
    \norm{y_n-y_m}\le\sum_{k=n+1}^m2^k\norm{x_{k+1}-x_k}\le\frac{1}{2^n}.
  \end{displaymath}
  We conclude that $(y_m)$ is norm Cauchy. By assumption, there exists
  $u\in X_+$ such that for every $n$ we have
  \begin{math}
    u\ge y_n\ge 2^n\abs{x_{n+1}-x_n}.
  \end{math}
  It follows that $(x_n)$ is $\norm{\cdot}_u$-Cauchy.

  \eqref{ncompl-Cauchy}$\Rightarrow$\eqref{ncompl-1X} Let
  $y\in\widetilde{X}$. There exists a sequence $(x_n)$ in $X$ which
  converges to $y$ in norm in~$\widetilde{X}$. It follows that $(x_n)$
  is norm Cauchy in~$X$. Using~\eqref{ncompl-Cauchy} and passing to a
  subsequence, we may assume that $(x_n)$ is uniformly Cauchy
  in~$X$. Therefore, $(x_n)$ converges uniformly to some $x$
  in~$\Xru$. Since $X$ is majorizing in $\Xru$, we can choose the
  regulator in~$X$. Since $(x_n)$ converges in norm to~$y$, we
  conclude that $y=x$.

  Implications \eqref{ncompl-1X}$\Rightarrow$\eqref{ncompl-ru}$\Rightarrow$\eqref{ncompl-maj}
  are straightforward.
\end{proof}

Combining Theorem~\ref{ncompl} with Proposition~\ref{ru-maj}, we get
the following corollary, which is somewhat analogous to
Theorem~5.29 in~\cite{Aliprantis:03}:

\begin{corollary}
  Suppose that $X$ is majorizing in~$\widetilde{X}$. Then $X$ is super
  order dense in~$\widetilde{X}$.
\end{corollary}

\begin{question}
  Can one replace ``norm Cauchy'' with ``norm null'' in
  Theorem~\ref{ncompl}\eqref{ncompl-obdd}? (After this paper was
  submitted, it was shown in~\cite{BB} that under the Continuum
  Hypothesis, this question has a negative answer.)
\end{question}

In Example~\ref{ex:finL0}, $\Xru$ is the ideal generated
by $X$ in $L_p(\mu)$. Observe also that $L_p(\mu)$ is the completion
of $X$ under~$\norm{\cdot}_{L_p}$. This proposition motivates the
following question:

\begin{question}
  Characterize those normed lattices $X$ for which $\Xru$ is an ideal
  in~$\widetilde{X}$.
\end{question}

\begin{proposition}
  If $\Xru$ is an ideal in $\widetilde{X}$ then $X$ is order dense
  in~$\widetilde{X}$.
\end{proposition}

\begin{proof}
  Being order dense in~$\Xru$, $X$ is regular in~$\Xru$. Being an
  ideal in~$\widetilde{X}$, $\Xru$ is regular
  in~$\widetilde{X}$. Combining these two facts, we conclude that $X$
  is regular in~$\widetilde{X}$. By
  \cite[Theorem~5.29]{Aliprantis:03}, $X$ is order dense
  in~$\widetilde{X}$.
\end{proof}

\section{Ball-Hager's example.}

\subsection*{A sublattice whose ru-adherence fails to be ru-closed.}

As observed in Example~\ref{ex:adh-non-cl}, the uniform adherence of a
set need not be uniformly closed. It has been a long-standing open
question whether the same is true for sublattices, see,
e.g.,~\cite[p.~239]{Quinn:75}. That is, if $Y$ is a sublattice of a
vector lattice~$X$, do we have $\uadh{Y}{1}=\ucl{Y}$?
In~\cite{BallHager}, the authors provide a counterexample to this
conjecture. We present a simplified version of their example. It is an
example of a vector lattice $X$ such that $\uadh{X}{1}\ne\ucl{X}$
in~$\Xru$ (and, therefore, in~$X^\delta$).

Let $P=[0,1]\setminus\mathbb Q$. By identifying every function in
$C[0,1]$ with its restriction to~$P$, we may view $C[0,1]$ as a
sublattice of $C(P)$. Note that $C(P)$ is uniformly complete by
Example~\ref{COmega-ucompl}.  For every $r\in[0,1]\cap\mathbb Q$, we define a
continuous function $f_r\colon[0,1]\to[0,\infty]$ via
$f_r(t)=\frac{1}{\abs{t-r}}$. Again, identifying $f_r$ with its
restriction to~$P$, we may view it as an element of $C(P)$. Let
$X\subseteq C(P)$ be defined as follows:
\begin{displaymath}
  X=C[0,1]+\Span\bigl\{f_r\mid r\in[0,1]\cap \mathbb Q\bigr\}.
\end{displaymath}
Every $f\in X$ admits unique expansion of the form
$f=v+\alpha_1f_{r_1}+\dots+\alpha_nf_{r_n}$, where $v\in C[0,1]$,
$r_1,\dots,r_n$ are distinct points in $[0,1]\cap\mathbb Q$, and
$\alpha_1,\dots,\alpha_n\in\mathbb R\setminus\{0\}$. We call
$r_1,\dots,r_n$ the singularities of~$f$.

\begin{lemma}
  $X$ is a sublattice of $C(P)$.
\end{lemma}

\begin{proof}
  Let $f\in X$. We can write $f=v+g-h$, where $v\in C[0,1]$,
  $g=\alpha_1f_{r_1}+\dots+\alpha_nf_{r_n}$,
  $h=\beta_1f_{s_1}+\dots+\beta_mf_{s_m}$,
  $r_1,\dots,r_n,s_1,\dots,s_m$ are distinct points in
  $[0,1]\cap\mathbb Q$, and $\alpha_1,\dots,\alpha_n$,
  $\beta_1,\dots,\beta_m>0$. We view $f$ as a continuous function from
  $[0,1]$ to the extended real line~$\overline{\mathbb R}$; let
  $\abs{f}$ be the point-wise modulus of~$f$. Note that $f$ agrees
  with $\abs{f}$ on a neighbourhood of each~$r_k$, and with $-f$ on a
  neighborhood of each~$s_k$. Put
  \begin{displaymath}
    u=
    \begin{cases}
      \abs{f}-g-h
      &\mbox{on }[0,1]\setminus\{r_1,\dots,r_n,s_1,\dots,s_m\}\\
      v-2h
      &\mbox{on }\{r_1,\dots,r_n\}\mbox{ and}\\
      -v-2g
      &\mbox{on }\{s_1,\dots,s_m\}.
    \end{cases}
  \end{displaymath}
  It can be easily verified that $u\in C[0,1]$. On~$P$, we have
  $u+g+h=\abs{f}$; this yields $\abs{f}\in X$.
\end{proof}

It now follows from Theorem~\ref{Xru-int-Z} that $\Xru$ may be viewed
as a sublattice of $C(P)$. As before, we write $\uadh{X}{1}$ and
$\uadh{X}{2}$ for the first and the second ru-adherences of $X$
in~$\Xru$. It suffices to show that $\uadh{X}{1}\ne\uadh{X}{2}$.

\begin{lemma}\label{fin-disc}
  Every $f\in \uadh{X}{1}$ is the restriction to $P$ of a function on $[0,1]$
  with only finitely many discontinuities.
\end{lemma}

\begin{proof}
  There is a sequence $(g_n)$ in $X$ such that $g_n\goesu f$
  in~$\Xru$, with a regulator $h\in\Xru_+$. Since $X$ is majorizing
  in~$\Xru$, we may assume that $h\in X$. Let $F_1$ be the set of the
  singularities of~$h$. Passing to a tail, we may also assume that
  $\abs{g_n-g_1}\le h$ for all~$n$. It follows that all $g_n$'s have
  the same set of singularities (namely, those of~$g_1$)
  outside~$F_1$; denote it~$F_2$. Put $F=F_1\cup F_2$. Then $h$ and
  all $g_n$'s extend continuously to $[0,1]\setminus F$; let us denote
  the extensions by $h'$ and~$g_n'$.

  For every $\varepsilon>0$ there exists $n_0$ such that for all
  $m,n\ge n_0$ we have $\abs{g_n-g_m}<\varepsilon h$. It follows that
  $\abs{g'_n-g'_m}<\varepsilon h'$ in $C\bigl([0,1]\setminus
  F\bigr)$. Since the ideal $I_{h'}$ in
  $C\bigl([0,1]\setminus F\bigr)$ is complete with respect to
  $\norm{\cdot}_{h'}$, there exists
  $f'\in C\bigl([0,1]\setminus F\bigr)$ such that $g'_n\goesu f'$; we
  can still use $h'$ as a regulator. It is now clear that $f$ is the
  restriction of $f'$ to~$P$.
\end{proof}

\begin{lemma}\label{char}
  Let $g$ be a bounded function on $[0,1]$, which is continuous except
  at finitely many rational points. Then $g_{|P}\in \uadh{X}{1}$.
\end{lemma}

\begin{proof}
  It is easy to see that $g$ may be expressed as a sum of finitely
  many bounded functions such that each
  of them is continuous except at one rational point. Hence, without
  loss of generality, $g$ has a single discontinuity at some
  $r\in[0,1]\cap\mathbb Q$. We may also assume that
  $\norm{g}_\infty\le 1$. For every~$n$, find $g_n\in C[0,1]$ such
  that $g_n$ agrees with $g$ outside $(r-\frac1n,r+\frac1n)$ and
  $\norm{g_n}_\infty\le 1$. Then
  \begin{math}
    \abs{g_n-g}\le 2\one_{[r-\frac1n,r+\frac1n]}\le\frac2n f_r.
  \end{math}
  It follows that $g_n\goesur{f_r}g$, hence $g\in \uadh{X}{1}$.
\end{proof}

\begin{lemma}
  There exists $g\in \uadh{X}{2}\setminus \uadh{X}{1}$.
\end{lemma}

\begin{proof}
  Let $A_n=[\frac{1}{2n+1},\frac{1}{2n}]\cap P$. Let $g_m$ be
  the characteristic function of $\bigcup_{n=1}^mA_n$ for every~$m$;
  let $g$ be the characteristic function of
  $\bigcup_{n=1}^\infty A_n$. Then $g_m\in \uadh{X}{1}$ for every $m$ by
  Lemma~\ref{char}, while $g\notin \uadh{X}{1}$ by Lemma~\ref{fin-disc}. It is
  easy to see, however, that $g_m\goesur{f_0} g$, so that $g\in \uadh{X}{2}$.
\end{proof}

This completes the proof of the example.

\subsection*{A uniformly closed sublattice which is not a kernel of an order
  bounded operator}
The preceding example can be further developed to
answer in the negative Question~4.8 in~\cite{Bilokopytov:24T}: Is
every uniformly closed sublattice the kernel of an order bounded
operator. The following construction is based on~\cite{Ball:99}.

Let $(r_n)$ be an enumeration of $\mathbb Q\cap[0,1]$. For each
$t\in P$, define
\begin{displaymath}
  g(t)=\sum_{n=1}^\infty\frac{1}{2^n}\sin\frac{1}{\abs{t-r_n}}.
\end{displaymath}
It is easy to see that this series converges uniformly in $C(P)$ with
regulator~$\one$. By Lemma~\ref{char}, partial sums of this series are
in $\uadh{X}{1}$; it follows that $g\in\uadh{X}{2}$. It is easy to see
that $g$ cannot be extended to a continuous function on $[0,1]$ or
even on any open subinterval of $[0,1]$. By Lemma~\ref{fin-disc}, we
have $g\notin \uadh{X}{1}$. It also follows that if $h$ is a function
on $[0,1]$ with only finitely many discontinuities then it cannot
agree with $g$ on any non-empty open set. Let $Y$ be the sublattice of
$C(P)$ generated by $X$ and~$g$.

\begin{lemma}\label{fin-disc-X}
  If $h\in Y$ is a restriction to $P$ of a function on $[0,1]$
  with finitely many discontinuities then $h\in X$.
\end{lemma}

\begin{proof}
  We may view $h$ as a function on $[0,1]$ with finitely many
  discontinuities. Without loss of generality, all discontinuities of
  $h$ are non-removable. We will prove the lemma by
  induction on the number of discontinuities. If $h$ has no
  discontinuities then $h\in X$ by the definition of~$X$.

  Since $h\in Y$ and $Y$ is generated by $X$ and~$g$, one can write
  \begin{displaymath}
    h=h_1-h_2,\quad\mbox{where}\quad
    h_1=\bigvee_{i=1}^n(x_i+\alpha_i g)\quad\mbox{and}\quad
    h_2=\bigvee_{j=1}^m(y_j+\beta_j g)
  \end{displaymath}
  for some $x_1,\dots,x_n,y_1,\dots,y_m$ in $X$ and
  $\alpha_1,\dots,\alpha_n,\beta_1,\dots,\beta_m$ in~$\mathbb
  R$. Without loss of generality, we assume that this decomposition is
  ``optimal'' in the following ways. First, we assume that
  $\alpha_i\ne\alpha_j$ whenever $i\ne j$ because
  $(x_i+\alpha g)\vee(x_j+\alpha g)=x_i\vee x_j+\alpha g$. Similarly,
  we assume that $\beta_i\ne\beta_j$ whenever $i\ne j$. Second, we
  assume that there are no redundant terms in the decomposition in the
  sense that removing any of $x_i+\alpha_i g$ or $y_j+\beta_jg$ terms
  from the decomposition will result in a different function.

  Let $Z$ be the set of all $f$ in $C(P)$ such that for every
  $r\in\mathbb Q\cap(0,1)$ and every sequence $(t_n)$ in
  $\mathbb R\setminus\mathbb Q$ with $t_n\to 0$ we have
  $f(r+t_n)-f(r-t_n)\to 0$. It is easy to see that $Z$ is a linear
  subspace of $C(P)$. It is, actually, a sublattice since $f\in Z$
  implies $\abs{f}\in Z$. Observe that $C[0,1]\subseteq Z$, $f_r\in Z$
  for every $r\in[0,1]\cap\mathbb Q$, and $g\in Z$. It follows that
  $Y\subseteq Z$. Hence, $h\in Z$.

  Let $q$ be a discontinuity of~$h$. We claim that then
  $\lim\limits_{s\to q}h(s)=\pm\infty$. Suppose not. It follows from
  $h\in Z$ that if $h$ has either a left or a right limit at $q$ then
  it has a limit at $q$, which contradicts the assumption that $q$ is
  a non-removable discontinuity. Hence, $h$ cannot have a left or a
  right limit at~$q$. Then $h$ must have at least two limit values,
  say, $a$ and $b$ with $a<b$, as it approaches $q$ from one side,
  say, from the right. Since $h$ has finitely many discontinuities, it
  is continuous on $(q,q+\delta)$ for all sufficiently small
  positive~$\delta$, hence $h(q,q+\delta)$ must contain $(a,b)$. It
  follows that $h$ has infinitely many limit values as it
  approaches~$q$. We will show that this leads to a contradiction.

  For each $i=1,\dots,n$, let
  $U_i=\bigl\{t\in P\mid h_1(t)=(x_i+\alpha_ig)(t)\bigr\}$.  Since
  both functions are in $C(P)$, $U_i$ is closed in~$P$. Clearly,
  $P=\bigcup_{i=1}^nU_i$. By the optimality assumption, $U_i$ is not
  contained in the union of the rest of $U_j$'s because otherwise,
  removing the term $x_i+\alpha_i g$ from the decomposition of $h$
  will result in the same function~$h$. Similarly, for every
  $j=1,\dots,m$, we put
  $V_j=\bigl\{t\in P\mid h_2(t)=(y_j+\beta_jg)(t)\bigr\}$; then $V_j$
  is closed, $P=\bigcup_{j=1}^mV_j$, and $V_j$ is not covered by
  $\bigcup_{i\ne j}V_i$.

  For any $i\ne j$, $x_i+\alpha_ig$ and $x_j+\alpha_jg$ agree on
  $U_i\cap U_j$ because both agree with $h_1$ there. It follows that
  $x_i-x_j=(\alpha_j-\alpha_i)g$ on $U_i\cap U_j$. Since
  $\alpha_j-\alpha_i\ne 0$, we conclude that $U_i\cap U_j$ is nowhere
  dense. So $U_i$'s form a partition of $P$ up to nowhere dense
  overlaps. Same for~$V_j$'s.

  Furthermore, on $U_i\cap V_j$, $h=(x_i+\alpha_ig)-(y_j+\beta_jg)$,
  hence $(\alpha_i-\beta_j)g=h+y_j-x_i$ there. If $\alpha_i\ne\beta_j$
  then $g$ agrees on $U_i\cap V_j$ with the restriction to
  $U_i\cap V_j$ of a function on $[0,1]$ with finitely many
  discontinuities; it follows that $U_i\cap V_j$ is nowhere
  dense. This means that the two partitions are almost aligned: for
  every $i$ there is exactly one $j$ for which $U_i\cap V_j$ is not
  nowhere dense. In particular, $n=m$. After relabelling, we may
  assume that $\alpha_i=\beta_i$ and $U_i\cap V_j$ is nowhere dense
  whenever $i\ne j$. It follows from $U_i\setminus
  V_i\subseteq\bigcup_{j\ne i}(U_i\cap V_j)$ that $U_i\setminus V_i$ is
  nowhere dense. Furthermore,
  \begin{displaymath}
    P=\bigcup_{i=1}^nU_i=\Bigl(\bigcup_{i=1}^n(U_i\cap V_i)\Bigr)
    \cup\Bigl(\bigcup_{i=1}^n(U_i\setminus V_i)\Bigr),
  \end{displaymath}
  so that
  \begin{math}
    P\setminus\Bigl(\bigcup_{i=1}^n(U_i\cap V_i)\Bigr)
    \subseteq\Bigl(\bigcup_{i=1}^n(U_i\setminus V_i)\Bigr).
  \end{math}
  Since the set on the left is open and the set on the right is
  nowhere dense, we conclude that the former set is empty and, therefore,
  \begin{math}
    P=\bigcup_{i=1}^n(U_i\cap V_i).
  \end{math}

  Suppose that $v$ is a (finite) limit value of $h$ at~$q$. Then there
  is a sequence $(s_n)$ in $P$ such that $s_n\to q$ and $h(s_n)\to
  v$. After passing to a subsequence, $(s_n)$ is contained in
  $U_k\cap V_k$ for some~$k$. On $U_k\cap V_k$, $h$ agrees with
  $x_k-y_k$; it follows that $v$ is a limit value of $x_k-y_k$
  at~$q$. Since $x_k-y_k\in X$, the only discontinuities of $x_k-y_k$
  are poles; it follows that $x_k-y_k$ is continuous at $q$ and,
  therefore,
  \begin{math}
    v=\lim_nh(s_n)=\lim_n(x_k-y_k)(s_n)=(x_k-y_k)(q).
  \end{math}
  We conclude that $v$ belongs to
  \begin{math}
    \bigl\{(x_i-y_i)(q)\mid i=1,\dots,k\bigr\}.
  \end{math}
  Hence, $h$ has only finitely many limit values at~$q$, which is a
  contradiction. This completes the proof that
  $\lim\limits_{s\to q}h(s)=\pm\infty$.

  Again, $q$ belongs to the closure of $U_k\cap V_k$ in $[0,1]$ for
  some $k=1,\dots,n$, and $h$ agrees with $x_k-y_k$ on $U_k\cap
  V_k$. It follows that $x_k-y_k$ has a singularity at~$q$. Therefore,
  $x_k-y_k-\gamma f_q$ is continuous at $q$ for some~$\gamma$. Put
  $h'=h-\gamma f_q$. Clearly, $h'\in Y$, $h'$ is discontinuous at all
  other discontinuities of $h$ and continuous where $h$ is continuous.
  We claim that $h'$ is continuous at~$q$. Indeed if $h'$ is
  discontinuous at $q$ then, by the preceding argument,
  $\lim\limits_{s\to q}h'(s)=\pm\infty$.  On the other hand, $h'$
  agrees with $x_k-y_k-\gamma f_q$ on $U_k\cap V_k$, so that
  $\lim_{s\to q}(x_k-y_k-\gamma f_q)(q)$ is a limit value of $h'$
  at~$q$. But this is a finite number because $x_k-y_k-\gamma f_q$ is
  continuous at~$q$. This contradiction proves that $h'$ is continuous
  at~$q$.

  Hence, $h'$ has one fewer discontinuity than~$h$. By the induction
  hypothesis, $h'\in X$. Then $h\in X$.
\end{proof}

\begin{corollary}\label{XYF}
  $X=\uadh{X}{1}\cap Y$.
\end{corollary}

\begin{proof}
  It is obvious that $X\subseteq \uadh{X}{1}\cap Y$. The other inclusion follows
  from Lemmas~\ref{fin-disc} and~\ref{fin-disc-X}.
\end{proof}

\begin{corollary}\label{FruXru}
  $X$ is uniformly closed in $Y$ and $\Yru=\Xru$.
\end{corollary}

\begin{proof}
  It follows from $g\in \uadh{X}{2}$ that
  $X\subseteq Y\subseteq \uadh{X}{2}\subseteq\ucl{X}=\Xru$. By
  Corollary~\ref{hier}, we have $\Yru=\Xru$.  Suppose that $x\in Y$
  and $(x_n)$ in $X$ are such that $x_n\goesu x$ in~$Y$. Then
  $x_n\goesu x$ in~$\Xru$, hence $x\in \uadh{X}{1}$. By
  Corollary~\ref{XYF}, we have $x\in X$.
\end{proof}

We can now complete the example. Let $H$ be an arbitrary
vector lattice; suppose that $T\colon Y\to H$ is an order bounded
operator such that $X\subseteq\ker T$. Consider
$T^{\rm ru}\colon\Yru\to H^{\rm ru}$. By
Corollary~\ref{FruXru}, we may view it as
$T^{\rm ru}\colon X^{\rm ru}\to H^{\rm ru}$. By uniqueness of
extension, we conclude that $T^{\rm ru}=0$ and, therefore, $T=0$.

\begin{question}
  Is every uniformly closed sublattice of a uniformly complete vector lattice
  the kernel of an order bounded operator? Some partial answers to
  this question can be found in~\cite{Hager:15}.
\end{question}

\section*{Acknowledgements.}
The authors would like to thank V.~Bohdanskyi for valuable discussions
about Example~\ref{ex:pos-inj}. We thank M.~Amine Ben Amor for
pointing our attention to~\cite{BallHager}. We also thank Y.~Azouzi for
helpful comments.

\end{document}